\documentclass[11pt]{amsart}
\usepackage{graphicx}
\usepackage{sseq}
\usepackage{amsthm}
\usepackage{amsmath}
\usepackage{amssymb}
\usepackage{sseq}
\usepackage[pagebackref]{hyperref}
\usepackage[curve]{xypic}
\usepackage[left=3.2cm,right=3.2cm]{geometry} 
\usepackage{tikz}
  \usetikzlibrary{positioning, matrix,arrows,decorations.pathmorphing}

\newtheorem{thm}{Theorem}[section]
\newtheorem*{thm*}{Theorem}
\newtheorem{theorem}[thm]{Theorem}
\newtheorem{lemma}[thm]{Lemma}

\newtheorem{cor}[thm]{Corollary}
\newtheorem{corollary}[thm]{Corollary}
\newtheorem{prop}[thm]{Proposition}
\newtheorem{proposition}[thm]{Proposition}
\newtheorem{claim}[thm]{Claim}

\newtheorem*{conjecture*}{Conjecture}

\newtheorem{question}[thm]{Question}
\newtheorem*{question*}{Question}

\theoremstyle{definition}
\newtheorem{defi}[thm]{Definition}
\newtheorem{definition}[thm]{Definition}
\newtheorem{notation}[thm]{Notation}

  \newtheorem*{example*}{Example}

\theoremstyle{remark}

\newtheorem{remark}[thm]{Remark}

\newcommand{\G}{\mathbb{G}}
\newcommand{\Z}{\mathbb{Z}}
\newcommand{\F}{\mathbb{F}}

\newcommand{\N}{\mathbb{N}}

\newcommand{\R}{\mathbb{R}}
\newcommand{\C}{\mathbb{C}}

\newcommand{\cp}{\mathbb{CP}}

\newcommand{\m}[1]{\underline{#1}}
\newcommand{\mZ}{\m{\Z}}
\newcommand{\mpi}{\m{\pi}}
\newcommand{\mM}{\m{M}}

\renewcommand{\AA}{\mathcal{A}}

\newcommand{\CC}{\mathcal{C}}

\newcommand{\EE}{\mathcal{E}}
\newcommand{\FF}{\mathcal{F}}
\newcommand{\GG}{\mathcal{G}}

\newcommand{\LL}{\mathcal{L}}
\newcommand{\PP}{\mathcal{P}}

\newcommand{\XX}{\mathcal{X}}

\newcommand{\YY}{\mathcal{Y}}

\newcommand{\OO}{\mathcal{O}}

\newcommand{\MM}{\mathcal{M}}

\newcommand{\MMb}{\overline{\mathcal{M}}}

\newcommand{\tensor}{\otimes}

\DeclareMathOperator{\pt}{pt}

\DeclareMathOperator{\Sp}{Sp}

\DeclareMathOperator{\Spec}{Spec}

\DeclareMathOperator{\sgn}{sgn}

\DeclareMathOperator{\modules}{\text{-}mod}
\DeclareMathOperator{\id}{id}

\DeclareMathOperator{\sm}{\wedge}
\DeclareMathOperator{\colim}{colim}

\DeclareMathOperator{\Ext}{Ext}
\DeclareMathOperator{\Ho}{Ho}

\DeclareMathOperator{\coker}{coker}
\DeclareMathOperator{\Hom}{Hom}

\DeclareMathOperator{\hocolim}{hocolim}
\DeclareMathOperator{\QCoh}{QCoh}

\DeclareMathOperator{\Pic}{Pic}
\DeclareMathOperator{\ind}{ind}

\DeclareMathOperator{\res}{res}
\DeclareMathOperator{\tr}{tr}

\def\co{\colon\thinspace}
\mathchardef\mhyphen="2D

\usepackage{stmaryrd}
\usepackage{lscape}
\newcommand{\ab}{\bar{a}}
\newcommand{\vb}{\bar{v}}
\newcommand{\cO}{\mathcal O}
\newcommand{\mA}{\m{A}}
\DeclareMathOperator{\Sym}{Sym}

\begin{document}
\title{The \texorpdfstring{$C_2$}{C2}-spectrum \texorpdfstring{$Tmf_1(3)$}{Tmf13} and its invertible modules}
\author{Michael A. Hill \and Lennart Meier}
\thanks{The first author was supported by NSF DMS-1307896 and the Sloan Foundation. The second author was supported by DFG SPP 1786.}

\begin{abstract}
We explore the $C_2$-equivariant spectra $Tmf_1(3)$ and $TMF_1(3)$. In particular, we compute their $C_2$-equivariant Picard groups and the $C_2$-equivariant Anderson dual of $Tmf_1(3)$. This implies corresponding results for the fixed point spectra $TMF_0(3)$ and $Tmf_0(3)$. Furthermore, we prove a Real Landweber exact functor theorem.
\end{abstract}

\maketitle

\section{Introduction}
The spectrum $TMF$ of topological modular forms comes in many variants. While $TMF$ itself arises from the moduli stack of elliptic curves $\MM_{ell}$, there is also a spectrum $Tmf$ associated with the compactification $\MMb_{ell}$. Finally, $tmf$ is defined as the connective cover of $Tmf$. It has been the spectrum $tmf$ and its cohomology that have been so far most relevant to applications (see e.g.\ \cite{BHHM08} and \cite{B-P04} for applications to generalized Toda--Smith complexes and \cite{AHR10}, \cite{H-M02} and \cite{Hil09} for applications to string bordism).

It is often simpler to work with topological modular forms with level structures. Among the many possibilities, the most relevant for us will be $TMF_1(n)$ and $TMF_0(n)$ corresponding to the moduli stacks $\MM_1(n)$ and $\MM_0(n)$. The former stack classifies elliptic curves with a chosen point of exact order $n$ and the latter elliptic curves with a chosen subgroup of order $n$. Note that for $n\geq 2$, the spectrum $TMF_1(n)$ is Landweber exact, while $TMF_0(n)$ is not in general, as will be explained in Section \ref{sec:Basics}.

Besides providing simpler variants of $TMF$, there are several reasons to care about $TMF$ with level structures. First, we mention the $Q(l)$-spectra defined by Behrens (\cite{Beh06}, \cite{B-O12}), which are built from $TMF$ with level structures and provide approximations of the $K(2)$-local sphere. Second, as shown in \cite{BOSS}, there is an injective map
\[\pi_*TMF \sm TMF \to \prod_{i\in\Z, j\geq 0} \pi_*TMF_0(3^j) \times \pi_*TMF_0(5^j),\]
important in the study of cooperations of $TMF$ and $tmf$. As a last point, we mention that Lurie defines in \cite{Lur07} a sheaf of $E_\infty$-ring spectra on the $n$-torsion of the universal elliptic curve over $\MM_{ell}$ whose global sections provide the value of $C_n$-equivariant $TMF$ at a point; if we invert $n$, this can be analyzed in terms of the $TMF_1(k)$ for $k|n$.

In each of these cases, it would be interesting to have compactified and connective variants. As a first step, Lawson and the first-named author overcame in \cite{HL13} certain technical obstacles to define $E_\infty$-ring spectra $Tmf_1(n)$ and $Tmf_0(n)$ corresponding to the compactified moduli stacks $\MMb_1(n)$ and $\MMb_0(n)$. One can then define $tmf_1(n)$ and $tmf_0(n)$ as the connective covers of these spectra and they form good connective models for $TMF_1(n)$ and $TMF_0(n)$ if $n$ is small. The aim of this article is to explore these spectra in the case $n=3$ with methods from Real homotopy theory.

Real homotopy theory is the study of genuinely equivariant $C_2$-spectra, also sometimes known as Real spectra. The theory has its origins in Atiyah's article \cite{Ati66} on Real K-theory and came to new prominence through the work of Hu--Kriz \cite{H-K01} and the work of Hill--Hopkins--Ravenel on the Kervaire Invariant One problem \cite{HHR09}.

The spectra $TMF_1(3)$ and $Tmf_1(3)$ inherit $C_2$-actions from an algebro-geometrically defined $C_2$-action on $\MMb_1(3)$. We will view them as cofree $C_2$-spectra (as explained in Sections \ref{sec:naive} and \ref{sec:Basics}) so that
\[
TMF_1(3)^{C_2} \simeq TMF_1(3)^{hC_2} \simeq TMF_0(3)
\]
and
\[
Tmf_1(3)^{C_2} \simeq Tmf_1(3)^{hC_2} \simeq Tmf_0(3).
\]
We define the $C_2$-spectrum $tmf_1(3)$ as the $C_2$-equivariant connective cover of $Tmf_1(3)$.

Mahowald and Rezk \cite{M-R09} have already computed the homotopy groups of $TMF_0(3)$ and a similar computation produces actually the $RO(C_2)$-graded $C_2$-equivariant homotopy groups of $tmf_1(3)$ and hence $TMF_1(3)$. Using this computation, we show that $tmf_1(3)$ has a Real orientation and is more precisely a form of $BP\R\langle 2\rangle$. This implies in particular that there exists a form of $BP\R\langle 2\rangle$ that is a strictly commutative $C_2$-spectrum, while it was not known before that there is a form of $BP\R\langle 2\rangle$ with any kind of ring structure. 

Moreover, we show that $TMF_1(3)$ is Real Landweber exact in the sense that
there is an isomorphism
\[
M\R_\bigstar(X) \tensor_{MU_{2*}} TMF_1(3)_{2*} \to TMF_1(3)_\bigstar X,
\]
natural in a $C_2$-spectrum $X$. Here $M\R_\bigstar(X)$ denotes the $RO(C_2)$-graded $C_2$-equivariant homology groups of $X$ with respect to the Real bordism spectrum $M\R$ and similarly for $TMF_1(3)_\bigstar X$.

As $\MMb_1(3)$ is proper over $\Spec \Z\big[\tfrac13\big]$, one expects a manifestation of Serre duality in $Tmf_1(3)$. A suitable duality to look for in the topological setting is \emph{Anderson duality}, an integral version of Brown--Comenetz duality. For example, Stojanoska computed in \cite{Sto12} that $Tmf[\frac12]$ is Anderson self dual in the sense that $I_{\Z[\frac12]}Tmf[\frac12] \simeq \Sigma^{21}Tmf[\frac12]$. We want to compute the $C_2$-equivariant Anderson dual $I_{\Z\big[\tfrac13\big]}Tmf_1(3)$ of $Tmf_1(3)$. While it is an easy calculation that non-equivariantly $I_{\Z\big[\tfrac13\big]}Tmf_1(3) \simeq \Sigma^9Tmf_1(3)$, this equivalence does \emph{not} hold $C_2$-equivariantly. We rather get the following:
\begin{thm*}
There is a $C_2$-equivariant equivalence
\[
I_{\Z\big[\tfrac13\big]}Tmf_1(3) \simeq \Sigma^{5+2\rho}Tmf_1(3),
\] where $\rho$ denotes the regular representation of $C_2$. It follows that
\[
I_{\Z\big[\tfrac13\big]}Tmf_0(3) \simeq (\Sigma^{5+2\rho}Tmf_1(3))^{hC_2}.
\]
\end{thm*}
Thus, the self-duality of $Tmf_0(3)$ is not fully apparent in the integer-graded homotopy groups \[\pi_*Tmf_0(3) \cong \pi_*^{C_2}Tmf_1(3),\] but only in the $RO(C_2)$-graded homotopy groups $\pi_\bigstar^{C_2}Tmf_1(3)$. Likewise, the resulting universal coefficient sequence uses $RO(C_2)$-graded homotopy groups. Indeed, for $C_2$-spectra $X$ the theorem implies a short exact sequence
\[
0 \to \Ext_{\Z[\tfrac13]}(R_{(a-6)+(b-2)\rho}^{C_2}(X), \Z[\tfrac13]) \to R_{C_2}^{a+b\rho}(X) \to \Hom_{\Z[\tfrac13]}(R_{(a-5)+(b-2)\rho}^{C_2}(X), \Z[\tfrac13]) \to 0.
\]
for $R = Tmf_1(3)$. We prove the theorem by an application of the slice spectral sequence. There has been similar work by Ricka \cite{Ric14} on Anderson duality of integral versions of Morava K-theory; our results have been obtained independently.

Next we turn to the topic of Picard groups. Given an $E_\infty$-ring spectrum $R$, its Picard group $\Pic(R)$ is defined as the group of invertible $R$-module spectra up to weak equivalence. From the perspective of \cite{B-N14}, these are the global twists of the associated cohomology theory and define a natural grading of $R$-homology groups. The Picard group was first introduced into stable homotopy theory by Hopkins; recent work of Mathew and Stojanoska \cite{M-S14} then significantly extended our toolbox for its computation. They show that all invertible $TMF$-modules are suspensions of $TMF$ so that $\Pic(TMF) \cong \Z/576$. In contrast, they show that $\Pic(Tmf)$ contains exotic elements that are not suspensions of $Tmf$ and compute $\Pic(Tmf) \cong \Z\oplus \Z/24$.

We will use their methods to understand $\Pic\big(TMF_0(3)\big)$ and $\Pic\big(Tmf_0(3)\big)$, but add a dash of equivariant homotopy theory. The maps
\[
Tmf_0(3) \to Tmf_1(3)
\]
 and
\[
TMF_0(3) \to TMF_1(3)
\]
are faithful $C_2$-Galois extensions in the sense of Rognes \cite[Thm 7.12]{MM15}. As explained in Section \ref{subsection:Generalities}, Galois descent then shows that
\[
\Pic\big(Tmf_0(3)\big) \cong \Pic_{C_2}\big(Tmf_1(3)\big)
\] 
and
\[
\Pic\big(TMF_0(3)\big) \cong \Pic_{C_2}\big(TMF_1(3)\big),
\]
where $\Pic_{C_2}\big(Tmf_1(3)\big)$ denotes the group of invertible $C_2$-module spectra over $Tmf_1(3)$ and similarly for $\Pic_{C_2}\big(TMF_1(3)\big)$. First we prove:
\begin{thm*}
Every invertible $TMF_0(3)$-module is an (integral) suspension of $TMF_0(3)$. Thus,
\[
\Pic_{C_2}\big(TMF_1(3)\big) \cong \Pic\big(TMF_0(3)\big) \cong \Z/48.
\]
\end{thm*}
The analogous theorem for $Tmf_0(3)$ is not true, but we have the following equivariant refinement:

\begin{thm*}
Every invertible $C_2$-equivariant $Tmf_1(3)$-module is an equivariant suspension $\Sigma^V Tmf_1(3)$, for an element $V \in RO(C_2)$. The corresponding homomorphism
\[
RO(C_2) \to \Pic_{C_2}\big(Tmf_1(3)\big),\quad V \mapsto \Sigma^V Tmf_1(3)
\]
is thus surjective and has kernel generated by $8-8\sigma$, for $\sigma$ the sign representation. Therefore,
\[
\Pic\big(Tmf_0(3)\big) \cong \Pic_{C_2}\big(Tmf_1(3)\big) \cong \Z\oplus \Z/8.
\]
\end{thm*}

We remark that invertible modules over $TMF$ with level structure occur in the study of equivariant $TMF$, for example those defined by representation spheres. We hope that our results on Picard groups may have relevance there. \\

We give a short overview of the structure of this article. Section \ref{sec:GSpectra} discusses preliminaries from equivariant homotopy theory. In particular, it is about the passage from spectra with $G$-action to genuine $G$-spectra and to their connective covers and how a commutative multiplication under this passages is preserved; furthermore, we discuss the $RO(G)$-graded homotopy fixed point spectral sequence and the slice spectral sequence. Section \ref{sec:Real} is about Real orientability and the Real Landweber exact functor theorem; it concludes with the definition and basic properties of forms of $BP\R\langle n\rangle$ and $E\R(n)$. Section \ref{sec:TMFandFriends} introduces the main characters $Tmf_0(3)$ and $Tmf_1(3)$ and their variants, discusses their relationship and computes the $RO(C_2)$-graded homotopy groups of $tmf_1(3)$; here we also present our applications to forms of $BP\R\langle 2\rangle$ and $E\R(2)$. Section \ref{sec:SliceAnderson} computes the slices of $Tmf_1(3)$ 
and applies this to compute its equivariant Anderson dual. Section \ref{sec:Picard} is about Picard groups, especially those of $TMF_0(3)$ and $Tmf_0(3)$. As a step, we prove a generalization of a result of \cite{B-R05} to give a conceptual computation of $\Pic\big(TMF_1(3)\big)$. Note that Sections \ref{sec:SliceAnderson} and \ref{sec:Picard} are independent and also independent of Section \ref{sec:Real}.

\subsection*{Conventions}
For a scheme $X$ with an action by a group scheme $G$, we denote by $X/G$ the \emph{stack} quotient. Furthermore, for a (pre)sheaf $\FF$ of spectra, $\pi_*\FF$ will always denote the \emph{sheafified} homotopy groups, i.e.\ the sheafification of $U \mapsto \pi_*(\FF(U))$. 

\subsection*{Acknowledgements}
We thank John Greenlees, Akhil Mathew and the anonymous referee for their comments on earlier versions of this article. We also thank the Hausdorff Institute for its hospitality during the time where the first version of this article was completed. 

\section{\texorpdfstring{$G$}{G}-spectra and equivariant homotopy}\label{sec:GSpectra}
After giving some basics on (genuine) $G$-spectra, we will treat in detail how to go from a spectrum with a $G$-action to a genuine $G$-spectrum, why this move preserves commutative multiplications and why the same is true for the passage to connective covers. After this, we will discuss the $RO(G)$-graded homotopy fixed point spectral sequence and the slice filtration. 

\subsection{Conventions on equivariant spectra}
We work in the category of genuine $G$-spectra for a finite group $G$, and our particular model will be orthogonal $G$-spectra. These were introduced by Mandell-May \cite{MM02}, though we draw heavily from \cite{HHR09} and also recommend \cite{SchEquiv} for a slightly different point of view on the same subject matter. In particular, a $G$-spectrum will always mean an orthogonal $G$-spectrum indexed on a complete $G$-universe, and morphisms are equivariant maps.

For each $H\subset G$ and for each $G$-spectrum $X$, we have stable homotopy groups
\[
\pi^H_n(X)=\colim_{V} [S^{V+\mathbb R^n},X(V)]^H,
\] 
where the colimit is taken over the finite dimensional representations of $G$ (or more simply, over the cofinal subsystem of sums of the regular representation), and for any representation $V$, the space $S^V$ is the $1$-point compactification. Recall finally that a map is a weak equivalence if it induces an isomorphism on these equivariant stable homotopy groups for all $H\subset G$. These are the weak equivalences in the standard model structures on $\Sp^G$ which give the genuine equivariant stable homotopy category; this extends the ordinary equivariant Spanier-Whitehead category described by Adams. 

Since we are considering the genuine model structure, the homotopy objects are naturally Mackey functor valued: for any two $G$-spectra $X$ and $Y$, the assignment
\[
T\mapsto [T_+\wedge X,Y]^G
\]
extends to an additive functor from the Burnside category of finite $G$-sets to abelian groups. In general, we will denote the obvious Mackey functor extension of classical objects like homotopy groups with an underline. In particular, we can rephrase the above condition on weak equivalences as simply that a map $f\colon X\to Y$ is a weak equivalence if it induces an isomorphism of homotopy Mackey functors $\m{\pi}_\ast X\to\m{\pi}_\ast Y$. 

\subsubsection{\texorpdfstring{$RO(G)$}{ROG}-grading and distinguished representations}
Since we are working genuine equivariantly, the representation spheres $S^V$ are elements of the Picard group of the homotopy category $\Ho(\Sp^G)$. In particular, all of our $\mathbb Z$-graded homotopy groups extend to $RO(G)$-graded homotopy groups, and similarly for Mackey functors, via the assignment
\[
T\mapsto [T_+\wedge S^V\wedge X,Y]^G.
\]
We will use this combined structure extensively, and when $X=S^0$, we will simply denote these groups by $\m{\pi}_V(Y)$. Note that to be precise, we have to choose (once and for all) for every element of $RO(G)$ an actual invertible $G$-spectrum and not just a class in the Picard group and by abuse of notation we will denote it also by $S^V$ for $V\in RO(G)$. Every such choice results in $\m{\pi}_{\bigstar}$ being a lax symmetric monoidal functor by \cite[Appendix A]{L-M06}.

We single out several representations.

\begin{notation}
\mbox{}
\begin{enumerate}
\item Let $\rho$ denote the regular representation of $G$.
\item Let $\bar{\rho}$ denote the quotient of $\rho$ by the trivial summand.
\item Let $\sigma$ denote the non-trivial $1$-dimensional real representation of $C_2$.
\end{enumerate}
\end{notation}

There are several distinguished homotopy classes of maps between representation spheres we shall need. If $V$ is a representation of $G$ with no fixed points, then let
\[
a_V\colon S^0\to S^V
\]
denote the inclusion of the fixed points into the $V$ sphere. This map is not null, and no iterate of it is null. However, its restriction to any subgroup $H$ such that $V^H\neq\{0\}$ is null-homotopic. This shows the following standard fact in equivariant stable homotopy theory.
\begin{lemma}\label{lem:Geometric}
Given a $G$-spectrum $X$, the geometric fixed points of $X$ can be computed as the $G$-fixed points of 
\[
X[a_{\bar{\rho}}^{-1}]=\hocolim\Big(X\xrightarrow{a_{\bar{\rho}}}\Sigma^{\bar{\rho}}X\xrightarrow{a_{\bar{\rho}}} \dots\Big).
\]
\end{lemma}
\begin{proof}
The homotopy colimit
\[
S^{\infty\bar{\rho}}=\hocolim\Big(S^0\xrightarrow{a_{\bar{\rho}}} S^{\bar{\rho}}\xrightarrow{a_{\bar{\rho}}}\dots\Big)
\]
is a model for the space $\tilde{E}\mathcal P$, where $\mathcal P$ is the family of proper subgroups of $G$. The geometric fixed points are computed by smashing $X$ with $\tilde{E}\mathcal P$ and taking fixed points, from which the result follows.
\end{proof}

\subsubsection{$G$-equivariant homology theories}
\begin{defi}
 Let $G$ be a finite group. An \emph{(ungraded) $G$-equivariant homology theory} is an exact functor $h_0 \colon \Ho(Sp^G) \to \mathrm{Ab}$ to the category of abelian groups (or any other abelian category) that sends (possibly infinite) coproducts to direct sums. 
\end{defi}
To such an ungraded homology theory we can associate an $RO(G)$-graded version as follows: For a given element $V \in RO(G)$, we consider the chosen invertible $G$-spectrum $S^V$ and define $h_V(X)$ as $h_0(S^{-V}\sm X)$. The resulting functor is also called an \emph{$RO(G)$-graded homology theory}. We will write $h_V$ for $h_V(S^0)$. 

For a $G$-spectrum $E$, we can define a $G$-equivariant homology theory by
$$X\mapsto E_0(X) = \pi_0^G(E\sm X)$$
and we clearly get also a natural isomorphism $E_{\bigstar}(X) \cong \pi^G_{\bigstar}(E\sm X)$ of the $RO(G)$-graded theories. 

\subsection{Passage from naive to genuine}\label{sec:naive}
The spectra which arise from algebraic geometry machines are almost never given to us as orthogonal $G$-spectra for some group $G$. Instead, they will be commutative ring spectra together with an action of $G$. There is a natural, homotopically meaningful way to prolong this to a genuine $G$-spectrum in a way which respects the multiplicative structure: passage to the cofree spectrum. It is easiest to explain this in two steps: extending a naive $G$-spectrum to a genuine one and then controlling the multiplicative structure.

\subsubsection{Additive structure}
Denote by $\Sp^G_{u}$ the category of orthogonal spectra with $G$-action. We consider an equivariant map $X\to Y$ to be an equivalence if it is a stable equivalence of the underlying non-equivariant orthogonal spectra. Since the inclusion of trivial representations of $G$ into a complete universe induces an equivalence 
\[I\colon \Sp^G_{u} \to \Sp^G\]
of categories \cite[Theorem V.1.5]{MM02}, we may consider any spectrum with a $G$-action as an orthogonal $G$-spectrum indexed on a complete universe. The functor $I$ is not homotopical, however.

In contrast, the functor
\[
\Sp^G_{u} \to \Sp^G, \quad X \mapsto IF(EG_+,X)
\]
preserves weak equivalences and defines therefore a derived functor for $I\colon \Sp^G_{u} \to \Sp^G$. Here $F(-,-)$ is understood to be the derived function spectrum so that it includes a fibrant replacement of $X$. We call $G$-spectra \emph{cofree} if they are up to weak equivalence in the image of $IF(EG_+,-)$. 

In particular, using the cofree functor $IF(EG_+,-)$, we may view any spectrum with a $G$-action as a genuine $G$-spectrum. We will use this to view $TMF_1(3)$ and $Tmf_1(3)$ as $C_2$-spectra.

\subsubsection{Multiplicative Concerns}
The homotopical behavior of the cofree functor on commutative ring spectra is most easily understood via an operadic approach using instead $E_\infty$-ring spectra. Let $\cO$ be an $E_\infty$ operad (for example, the linear isometries operad). As the model category of orthogonal spectra fulfills the monoid axiom by \cite[Theorem 12.1]{MMSS}, \cite[Thm 4]{Spitzweck} implies that the category of $\cO$-algebras in orthogonal spectra with $G$-action has a projective model structure. Thus, if $R$ is an $\cO$-algebra with $G$-action, there exists an $\cO$-algebra with $G$-action that is fibrant as a spectrum and weakly equivalent to $R$. 

The equivalence of categories $I$ above is strong symmetric monoidal, so in particular, it takes $\cO$-algebras to $\cO$-algebras in orthogonal $G$-spectra indexed on a complete universe. Here, the group $G$ acts trivially on the operad, so this is the prototypical example of a naive $N_\infty$ operad in the sense of \cite{B-H15}. Applying $IF(EG_+,-)$ takes $R$ to $IF(EG_+,R)$, which is an algebra over $F(EG_+,\cO)$ if $R$ is fibrant. However, this operad is a $G$-$E_\infty$ operad \cite[Theorem 6.25]{B-H15}. In particular, since the category of algebras over a $G$-$E_\infty$ operad is Dwyer--Kan equivalent to the category of equivariant commutative ring spectra by \cite[Theorem A.6]{B-H15}, we conclude the following.

\begin{theorem}\label{thm:cofree}
If $R$ is a commutative ring spectrum with a $G$-action via commutative ring maps, then $IF(EG_+,R)$ is an equivariant commutative ring spectrum. More precisely, one can functorially associate to $R$ an equivariant commutative ring spectrum $R'$ such that $R'$ and $IF(EG_+, R)$ are equivalent as $E_\infty$-algebras.
\end{theorem}

In particular, this will immediately imply that $TMF_1(3)$ and $Tmf_1(3)$ can actually be viewed as $C_2$-equivariant commutative ring spectra. Deducing a similar result for the equivariant connective cover of $Tmf_1(3)$ will require a simple result undoubtedly known to the experts. The proof is also standard; we include it for completeness. Before proceeding, recall the following result about the connectivity of symmetric powers.

\begin{lemma}\label{lem:Connectivity}
If $X$ is a $(k-1)$-connected equivariant spectrum with $k\geq 0$, then for all $n\geq 1$, 
\[
\Sym^n(X) = X^{\sm n}/\Sigma_n
\]
is also $(k-1)$-connected.
\end{lemma}
\begin{proof}
This follows from the weak equivalence
\[
E_G\Sigma_{n+}\wedge_{\Sigma_n} X^{\wedge n}\to\Sym^n(X).
\]
As in \cite[Prop B.171, 177]{HHR09}, we can reduce the statement of the lemma by this equivalence to the following statement: the $G$-spectrum $\bigwedge_{G/H} S^l$ is $(k-1)$-connected for all $H\subset G$ and all $l\geq k$. This is clear, as $\bigwedge_{G/H} S^l\cong S^{\ind_H^G l}$ and $\ind_H^G$ contains an $l$-dimensional trivial summand. 
\end{proof}

\begin{remark}If $X$ is $(k-1)$-connected for $k>1$, then we do not always get a bump in the connectivity of the symmetric powers as happens classically. For $n$ sufficiently large, the $n$th symmetric power is more highly connected than $X$, but for low values of $n$, they are often equally connected. The reason for this is the norm: if $[G:H]=m$, then there is a canonical homotopy class of maps 
\[
N_H^Gi_H^\ast X\to \Sym^m(X)
\]
coming from any inclusion of $G\times\Sigma_m/\Gamma$ into $E_G\Sigma_m$, where $\Gamma$ is the graph of the homomorphism $G\to\Sigma_m$ defining $G/H$ as a $G$-set \cite{HHR09}, \cite{B-H15}. In particular, if $\rho_{H}^G$ is the representation $Ind_H^G\mathbb R$ and if $\bar{\rho}_H^G$ is the quotient of $\rho_H^G$ by the trivial summand, then for any class $x\in\pi_k^G(X)$, we have an element
\[
a_{\bar{\rho}_H^G}^k N_H^G i_H^\ast(x)\in \pi_k^G\big(\Sym^m(X)\big).
\]
Checking on the case of spheres shows that these maps are generically non-trivial. This is the only complicating factor in the proof of the following theorem, since it means that the $k$th homotopy Mackey functor of the free commutative ring spectrum on something $(k-1)$-connected is strictly larger than 
\[
\m{\pi}_kS^0\oplus \m{\pi}_k(X).
\]
\end{remark}

\begin{theorem}\label{thm:connective}
If $R$ is a $G$-equivariant commutative ring spectrum, then there is a commutative ring spectrum structure on the $-1$-connected cover $r$ of $R$ such that the canonical map $r\to R$ is a map of commutative ring spectra.
\end{theorem}
\begin{proof}
We will inductively build a series of $(-1)$-connected commutative ring spectra $r^k$ over $R$ for $-1\leq k$ such that the induced map on homotopy group $\m{\pi}_j (r^k)\to\m{\pi}_k(R)$ is an isomorphism for $0\leq j\leq k$ (this condition is vacuous when $k=-1$). Let $r^{-1}$ denote the zero sphere, which maps to $R$ via the unit.

Assume that we have built $r^{k-1}\to R$ as above. We can assume that $r^{k-1}$ is cofibrant in the positive model structure on equivariant commutative rings as in \cite[Prop B.130]{HHR09}. To build $r^k$, we first choose a surjective map
\[
\bigoplus_{i\in I_k} \mA_{G/H_i}\to \m{\pi}_k(R), 
\]
where $\mA_{G/H_i}$ is the Mackey functor $\m{\pi}_k(G/H_{i+}\wedge S^k)$. Any such surjective map can be realized topologically as a map
\[
\bigvee_{i\in I_k} G/H_{i+}\wedge S^k\xrightarrow{j_k} R,
\]
and this induces a map of commutative ring spectra
\[
e_k=\mathbb P(\bigvee_{i\in I_k} G/H_{i+}\wedge S^k)\to R,
\]
where $\mathbb{P}$ denotes the free commutative ring spectrum functor. 
Smashing this with the map $r^{k-1}\to R$ gives a map
\[
e_k\wedge r^{k-1}\xrightarrow{J_k} R.
\]
This is the correct derived smash product by \cite[Prop 2.30]{HHR09}. The map $S^0\to e_k$ induces an isomorphism in homotopy groups through dimension $(k-1)$ by Lemma~\ref{lem:Connectivity}, and the K\"unneth spectral sequence of Lewis-Mandell \cite{L-M06} implies that the map $J_k$ induces an isomorphism in homotopy in dimensions between $0$ and $(k-1)$ and a surjection in dimension $k$.

At this point, the argument is classical. Let $F_k$ denote the fiber of $e_k\wedge r^{k-1}\to R$, and let $f_k$ denote the $(-1)$-connected cover of $F_k$. Since the map $e_k\wedge r^{k-1}\to R$ was a map of commutative ring spectra, the composite 
\[
Sym^n(f_k)\to Sym^n(F_k)\to e_k\wedge r^{k-1}\to R
\]
is null for all $n >0$. In particular, if we let $r^k$ denote the pushout in commutative ring spectra
\[
\xymatrix{
{\mathbb P(f_k)}\ar[r]\ar[d] & {e_k\wedge r^{k-1}}\ar[d] \\
{S^0}\ar[r] & {r^k,}
}
\]
then we have an extension of $J_k$ over $r^k$. Note that $r^k$ is actually equivalent to the \emph{derived} smash product $S^0 \sm_{\mathbb{P}(f_k)} (e_k\sm r^{k-1})$ because $S^0$ is a cofibrant $\mathbb{P}(f_k)$-module with respect to a monoidal model structure \cite[Prop B.137]{HHR09}. 

We have a cofiber sequence $\overline{\mathbb{P}}(f_k) \to \mathbb{P}(f_k) \to S^0$. Because $r^k$ is a retract of $e_k\sm r^{k-1}$, this induces short exact sequences
\[0 \to \m{\pi}_i(\overline{\mathbb{P}}(f_k)\sm_{\mathbb{P}(f_k)} (e_k\sm r^{k-1})) \to \m{\pi}_i(e_k\sm r^{k-1}) \to \m{\pi}_i(r^k) \to 0\]
for every $i\in\Z$. Since $\overline{\mathbb{P}}(f_k) \cong \bigvee_{n\geq 1} \mathrm{Sym}^n(f_k)$, Lemma \ref{lem:Connectivity} implies that $\overline{\mathbb{P}}(f_k)$ is $(k-1)$-connected. By the K\"unneth spectral sequence, $\overline{\mathbb{P}}(f_k)\sm_{\mathbb{P}(f_k)} (e_k\sm r^{k-1})$ is thus also $(k-1)$-connected. Thus, $\m{\pi}_i (e_k\sm r^{k-1}) \to \m{\pi}_ir^k$ is an isomorphism for $i\leq k-1$ and hence also $\m{\pi}_ir^k \to \m{\pi}_iR$. 

For the analysis of $\m{\pi}_k$, consider the diagram
\[\xymatrix{\m{\pi}_k(\overline{\mathbb{P}}(f_k))\ar[d] \ar[r] & \m{\pi}_k(f_k) \ar[d] \\
\m{\pi}_k(\overline{\mathbb{P}}(f_k)\sm_{\mathbb{P}(f_k)} (e_k\sm r^{k-1})) \ar[r]^-{\iota_*} & \m{\pi}_k(e_k\sm r^{k-1}) \ar[r]^-{(J_k)_*} & \m{\pi}_k(R) \\
 }
\]
We know that $\m{\pi}_k(f_k)$ surjects onto the kernel of $(J_k)_*$. Because $f_k$ is a summand of $\overline{\mathbb{P}}(f_k)$, also $\iota_*$ must surject onto the kernel of $(J_k)_*$. Thus, $\m{\pi}_kr^k \cong \coker(\iota_*)$ maps injectively into $\m{\pi}_kR$ and also surjectively because already $(J_k)_*$ is surjective. 
%
Now define $r$ as the colimit of the $r^k$. Clearly, $r$ is connective and the maps $r^k\to R$ extend to a map $r\to R$ that induces an isomorphism in $\pi_i$ for $i\geq 0$. 
\end{proof}

\subsection{The \texorpdfstring{$RO(G)$}{ROG}-graded homotopy fixed point spectral sequence}\label{sec:RO(G)HFPSS}
If $V$ is a virtual representation of $G$, then by tracing through the adjunctions, we see that
\[
\pi_{V}^G F(EG_{+},X)\cong [S^{V}\wedge EG_{+},X]^{G}\cong \pi_{0}^G F(EG_{+},S^{-V}\wedge X).
\]
For a $G$-spectrum $E$ denote by $\mathrm{HF}(E)$ the homotopy fixed point spectral sequence for $E$ (as constructed e.g.\ in \cite[Section 6]{Dug03}). Let $V_1,\dots, V_n$ be representatives of the isomorphism classes of non-trivial irreducible real representations of $G$. We define the $RO(G)$-graded homotopy fixed point spectral sequence for $E$ as
\[
\mathrm{HF}^{RO(G)}(E) = \bigoplus_{a_i\in\Z} \mathrm{HF}(a_1V_1+\cdots a_nV_n, E),
\]
where
\[
\mathrm{HF}(a_1V_1+\cdots a_nV_n, E) = \mathrm{HF}(S^{-a_1V_1}\sm \dots \sm S^{-a_nV_n} \sm E).
\]
We can use the twisting isomorphisms of the symmetric monoidal structure on $G$-spectra to define isomorphisms
\[
S^{-a_1V_1}\sm \dots \sm S^{-a_nV_n} \sm S^{-b_1V_1}\sm \dots \sm S^{-b_nV_n} \cong S^{-(a_1+b_1)V_1}\sm \dots \sm S^{-(a_n+b_n)V_n}.
\]
As in \cite{Dug03}, a multiplication $E\sm E \to E$ defines then multiplicative pairings
\[
\mathrm{HF}(a_1V_1 + \cdots a_nV_n,E) \tensor \mathrm{HF}(b_1V_1+ \cdots + b_nV_n, E) \to \mathrm{HF}((a_1+b_1)V_1+\cdots +(a_n+b_n)V_n, E).
\]
As explained in \cite[Appendix A]{L-M06}, we can choose the isomorphisms above so that this actually defines an associative and commutative multiplication on $\mathrm{HF}^{RO(G)}(E)$. We summarize in the following proposition.

\begin{proposition}
If $E$ is a $G$-spectrum with a multiplication up to homotopy, then there is a multiplicative $RO(G)$-graded spectral sequence
\[
E_{2}^{s,V}=H^{s}\big(G;\pi_{0}(S^{-V}\wedge E)\big)\Rightarrow \pi^G_{V-s} F(EG_{+},E).
\]
In particular, the Leibniz rule states that for elements $x\in E_{r}^{s,V}$ and $y\in E_{r}^{t,W}$ with $V=a_0+a_1V_1+\cdots + a_nV_n$, we have $d_r(xy) = d_r(x)y + (-1)^{a_0}xd_r(y)$.
\end{proposition}

Note that while the $RO(G)$-graded homotopy fixed point spectral sequence decomposes additively in infinitely many summands, we package them into one spectral sequence for the sake of a more efficient multiplicative presentation. In our later computations, our generating permanent cycles will sit in non-integral degrees.

\subsection{The \texorpdfstring{$C_{2}$}{C2}-equivariant slice filtration}\label{sec:Slices}
The $C_{2}$-equivariant slice filtration was introduced by Dugger in his study of Atiyah's Real $K$-theory. This was generalized by Hopkins, Ravenel, and the first author to arbitrary finite groups in the solution to the Kervaire Invariant One problem \cite{HHR09}. We will recall some of the basic properties here. A more detailed treatment can be found in \cite{HHR09} or \cite{SlicePrimer}.

\begin{proposition}[{\cite[Proposition 4.20 \& Lemma 4.23]{HHR09}}]\label{prop:OddSlices}
For any $C_{2}$-equivariant spectrum $E$, the odd slices are determined by the formula
\[
P_{2n-1}^{2n-1}(E)=\Sigma^{n\rho-1} H \m{\pi}_{n\rho-1} E.
\]
\end{proposition}

\begin{corollary}
If $R$ is a $C_2$-spectrum such that $\m{\pi}_{n\rho-1}R=0$, then all odd slices of $R$ vanish.
\end{corollary}

For the even slices, there is a similar formula involving homotopy Mackey functors of $E$.

\begin{definition}
If $\mM$ is a $C_{2}$ Mackey functor, let $P^{0}\mM$ denote the maximal quotient of $\mM$ in which the restriction map $\mM(C_{2}/C_{2})\to \mM(C_2/e)$ is injective.
\end{definition}

There are several equivalent formulations. One of which is to notice that we can build a Mackey functor out of the kernel of the restriction by declaring that the value at $C_{2}/C_{2}$ is the kernel of the restriction map and that the value at $C_{2}/\{e\}$ is trivial. The functor $P^{0}\mM$ is then the quotient of $\mM$ by this subMackey functor.

The second reformulation requires an endofunctor on Mackey functors.
\begin{definition}
If $T$ is a finite $C_{2}$-set and $\mM$ is a Mackey functor, then let $\mM_{T}$ be the Mackey functor defined by
\[
S\mapsto \mM(T\times S).
\]
\end{definition}
The restriction map defines a map of Mackey functors
\[
\mM\to \mM_{C_{2}},
\]
and $P^{0}\mM$ is simply the image of this map.

\begin{proposition}[{\cite[Cor 2.16]{SlicePrimer}}]\label{prop:EvenSlices}
For any $C_{2}$-equivariant spectrum $E$, the even slices are determined by the formula
\[
P_{2n}^{2n}(E)=\Sigma^{n\rho} H P^{0}\m{\pi}_{n\rho}(E).
\]
In particular, if $\m{\pi}_{n\rho}(E)$ is constant, we have
\[
P_{2n}^{2n}(E)=\Sigma^{n\rho} H\m{\pi_{2n}(E)}.
\]
\end{proposition}

Knowledge of the slices is important because of the slice spectral sequence
\[E_2^{s,t} = \m{\pi}_{t-s}P^t_tX \Rightarrow \m{\pi}_{t-s}X,\]
which we will always depict in Adams notation where $E_2^{s,t}$ is in the spot $(t-s, s)$. 

We need several Mackey functors. We will define them via a Lewis diagram, stacking the value of the Mackey functor at $C_{2}/C_{2}$ over that of $C_{2}/\{e\}$ and then drawing in the restriction map, the transfer map, and the action of the non-trivial element of the Weyl group.

\begin{definition}\label{def:Mackey}
Let $\m{G}$, $\mZ_{-}$ and $\mZ^*$ be the Mackey functors defined by
\[
\xymatrix{
{\mM(C_2/C_2):}\ar@(l,l)[d]_{res}& & {\Z/2}\ar@(l,l)[d] & &  {0}\ar@(l,l)[d] & & \Z \ar@(l,l)[d]^2 \\
{\mM(C_2/e):}\ar@(r,r)[u]_{tr} \ar@(dl,dr)[]_{\gamma}& & {0}\ar@(r,r)[u] \ar@(dl,dr)[] & & {\Z}\ar@(r,r)[u]\ar@(dl,dr)[]_{-} & & {\Z}\ar@(r,r)[u]^{1}\ar@(dl,dr)[]_{1}\\
{\mM\colon} & & {\m{G}} & & {\mZ_{-}} & & {\mZ^*}
}
\]
\end{definition}

\begin{lemma}\label{lem:HomotopyGroups}
If $X$ is a $C_{2}$-spectrum such that
\begin{enumerate}
\item $\m{\pi}_{n\rho-1} X=0$ for all $n$ and
\item $\m{\pi}_{n\rho} X=\mZ\otimes\pi_{2n} X$, where $\pi_{2n} X$ has no $2$-torsion,
\end{enumerate}
then we have
\begin{align*}
 \m{\pi}_{k\rho+1} X & = \m{G}\tensor_{\Z} \pi_{2k+2}X \\
 \m{\pi}_{k\rho} X&=\mZ\tensor_{\Z} \pi_{2k}X, \\
 \m{\pi}_{k\rho-1} X&=0, \text{ and } \\
 \m{\pi}_{k\rho-2} X&=\mZ_{-}\tensor_{\Z} \pi_{2k-2}X.
\end{align*}
\end{lemma}
\begin{proof}
To simplify notation, let $A_{k}={\pi}_{2k}X$, let $\m{A}_{k}=\m{\Z}\otimes A_{k}$, let $\m{A}_{k}^{-}=\m{\Z}_{-}\otimes A_{k}$, and let $B_{k}=\m{G}\otimes A_{k}$. By assumption,
we have $P_{2k-1}^{2k-1}X\simeq \ast$ and
\[
P_{2k}^{2k}X\simeq S^{k\rho}\wedge H(\m{A}_{k}).
\]
Smashing the slice tower for $X$ with $S^{-k\rho}$ gives the slice tower for $\Sigma^{-k\rho}\wedge X$, and this again has the property that the odd slices vanish and the even ones are of the above form. It therefore suffices to prove this for $k=0$. The homotopy Mackey functors in question are all especially simple, as they are in the region where the can be no differentials in the slice spectral sequence, as we will see. 

\begin{figure}\label{fig:E2}
\centering
\begin{sseq}[entrysize=.33in]{-3...2}{-2...2}
\ssdrop{\m{A}_{k}} \ssmove{1}{1} \ssdrop{\m{B}_{k+1}} \ssmove{1}{-1} \ssdrop{\m{A}_{k+1}^{-}} \ssmove{0}{2} \ssdrop{\m{B}_{k+2}} \ssmove{-2}{-2} \ssmove{-2}{0} \ssdrop{\m{A}_{k-1}^{-}}
\end{sseq}
\caption{The slice $E_2$ term for $S^{-k\rho}\wedge X$. This is Adams graded, with a group in position $(t-s,s)$ recording $\m{\pi}_{t-s} (P^{t}_{t} S^{-k\rho}\wedge X)$.}
\end{figure}
By the connectivity of the regular representation spheres, the $2m$th slice does not contribute to $\m{\pi}_i X$ for $i = -2,-1,0,1$ and $m<-2$ or $m>1$. Similarly, $\m{H}_{-2}(S^{-2\rho};\m{\Z}\otimes A_{k})=0$ for any abelian group $A_{k}$ (this is the essential part of the Gap Theorem in \cite{HHR09}), so the $-4$th slice does not contribute to these homotopy Mackey functors either. The cell-structures for representation spheres then show that the slice $E_{2}$ term has the form depicted in Figure~\ref{fig:E2}.

In particular, there is no room for differentials or extensions in the range considered, and the result follows.
\end{proof}

\section{Real orientations and Real Landweber exactness}\label{sec:Real}
In this section, we will first treat some basics about Real orientations. Then we will prove a Real version of the Landweber exact functor theorem, both in classical and in stack language. In the last subsection, we define what we mean by forms of $BP\R\langle n\rangle$ and $E\R(n)$ and apply the Real Landweber exact functor theorem to the latter.

\subsection{Basics}\label{sec:BasicsC2}
Given a $C_2$-spectrum $E\R$, we denote by $E\R_\bigstar(X)$ the value of the associated $RO(C_2)$-graded homology theory on a $C_2$-spectrum $X$ and we set $E\R_\bigstar = E\R_\bigstar(\pt)$. This is the value at $C_{2}/C_{2}$ of the associated Mackey functor valued homology.

\begin{defi}
A $C_2$-spectrum $E\R$ is \emph{even} if $\m{\pi}_{k\rho -1}E\R=0$ for all $k\in\Z$. It is called \emph{strongly even} if additionally $\m{\pi}_{k\rho}E\R$ is a constant Mackey functor for all $k\in\Z$, i.e.\ if the restriction 
$$\pi^{C_2}_{k\rho}E\R \to \pi^e_{k\rho}E\R\cong \pi^e_{2k}E\R$$
is an isomorphism.
\end{defi}
For example, by \cite[Theorem 4.11]{H-K01}, the Real bordism spectrum $M\R$ and $BP\R$ are strongly even (see also Appendix A of \cite{G-M16} for an alternative exposition). These $C_2$-spectra were introduced by Landweber \cite{Lan68} and Araki \cite{Ara79} and modern treatments can be found in \cite[Section 2]{H-K01} and \cite[Example 2.14]{SchEquiv}. 

Recall the following definition:
\begin{defi}
Let $X$ be a $C_2$-spectrum. A \emph{Real orientation} for $E\R$ is a class
\[
x \in E\R^\rho(\cp^\infty) =[\cp^\infty, S^\rho\wedge E\R]^{C_2},
\]
restricting to the class in $E\R^\rho(\cp^1) \cong [\cp^1, S^\rho\wedge E\R]^{C_2}$ corresponding to
\[1\in [S^{0},E\R]^{C_2}\cong [S^\rho, S^\rho\wedge E\R]^{C_2}\]
 under the (chosen) isomorphism $S^\rho = \cp^1$. Here, we view $\cp^n$ as a $C_2$-space via complex conjugation.
\end{defi}

By \cite[Theorem 2.25]{H-K01}, Real orientations of commutative $C_2$-ring spectra are in one-to-one correspondence with homotopy classes of maps $M\R \to E\R$ of $C_2$-ring spectra, where ring spectra are understood to be up to homotopy. Another point of view uses the notion of a Real vector bundle, i.e.\ a complex vector bundle $p\colon V\to X$ on a $C_2$-space together with an antilinear involution such that $p$ is $C_2$-equivariant. If $E$ is Real oriented, then every Real vector bundle carries a canonical $E$-orientation.  

\begin{lemma}\label{lem:RealOrientable}
Every even $C_2$-spectrum $E\R$ is Real orientable.
\end{lemma}
\begin{proof}
We have cofiber sequences
\[
S^{(n+1)\rho-1}\to \cp^n \to \cp^{n+1}.
\]
The long exact sequence in cohomology then shows that the map
\[
E\R^{\rho}(\cp^{n+1})\to E\R^{\rho}(\cp^{n})
\]
is surjective. The Milnor sequence gives the result.
\end{proof}

It is part of our philosophy that the Mackey functor $\mpi_{k\rho}$ behaves often much better than the integral Mackey functor $\mpi_{2k}$. The following is a weak version of a Whitehead theorem using $\mpi_{k\rho}$. We will formulate it in the language of equivariant homology theories as this will be convenient for our use in the Real Landweber exact functor theorem.

\begin{lemma}\label{lem:regrep}
Let $f\co E\R \to F\R$ be a natural transformation of $C_2$-equivariant homology theories. Denote the underlying homology theories by $E$ and $F$. Assume that f induces isomorphisms
\[E\R_{k\rho} \to F\R_{k\rho} \quad \text{and} \quad E_k \to F_k\]
for all $k\in\Z$. Assume furthermore that $E\R_{k\rho-1} \to F\R_{k\rho-1}$ is mono for all $k\in\Z$ (this is the case, for example, if $E\R_{k\rho-1} = 0$). Then $f$ is a natural isomorphism.
\end{lemma}
\begin{proof}
It is well known that it is enough to show that $E_k \to F_k$ and $E\R_k \to E\R_k$ are isomorphisms for all $k\in\Z$. As the former is true by assumption, it is in particular enough to show that $f_{a+b\sigma}\co E\R_{a+b\sigma} \to F\R_{a+b\sigma}$ is an isomorphism for all $a,b\in\Z$. This is true for $a=b$ again by assumption.

Smashing the cofiber sequence
\[
(C_2)_+ \to S^0 \to S^{\sigma}
\]
with $S^{a+b\sigma}$ gives the cofiber sequence
\[
(C_2)_+\sm S^{a+b\sigma} \to S^{a+b\sigma} \to S^{a+(b+1)\sigma}.
\]
We have a map between the associated long exact sequences:
\[
\xymatrix{
E_{a+b+1} \ar[r]\ar[d]^{\cong} & E\R_{a+(b+1)\sigma} \ar[r]\ar[d]^{f_{a+(b+1)\sigma}} & E\R_{a+b\sigma} \ar[r] \ar[d]^{f_{a+b\sigma}}& E_{a+b} \ar[r]\ar[d]^\cong & E\R_{(a-1)+(b+1)\sigma}\ar[d]^{f_{(a-1)+(b+1)\sigma}} \\
F_{a+b+1} \ar[r] & F\R_{a+(b+1)\sigma} \ar[r] & F\R_{a+b\sigma} \ar[r] & F_{a+b} \ar[r] & F\R_{(a-1)+(b+1)\sigma}
}
\]
The weak $5$-lemmas imply the following statements:
\begin{itemize}
\item[(M1)] If $f_{a+(b+1)\sigma}$ is mono, then $f_{a+b\sigma}$ is mono.
\item[(M2)] If $f_{(a+1)+b\sigma}$ is epi and $f_{a+b\sigma}$ is mono, then $f_{a+(b+1)\sigma}$ is mono.
\item[(E1)] If $f_{a+b\sigma}$ is epi, then $f_{a+(b+1)\sigma}$ is epi.
\item[(E2)] If $f_{(a-1)+(b+1)\sigma}$ is mono and $f_{a+(b+1)\sigma}$ is epi, then $f_{a+b\sigma}$ is epi.
\end{itemize}

\noindent This implies the following four statements in turn,
\begin{enumerate}
\item By hypothesis $f_{a+a\sigma} = f_{a\rho}$ is epi for all $a$, and hence repeated application of $E1$ shows that $f_{a+b\sigma}$ is epi for $b\geq a$.
\item By hypothesis $f_{(a-1)+a\sigma} = f_{a\rho-1}$ is mono for all $a$, and hence $f_{a+b\sigma}$ is mono for $b \leq a+1$ by repeated application of M1.
\end{enumerate}
Note that the regions in which $f_{a+b\sigma}$ is epi and mono overlap in two diagonals, allowing us to proceed.
\begin{enumerate}
 \setcounter{enumi}{2}
\item By repeated application of E2 we conclude that $f_{a+b\sigma}$ is epi for all $a,b$.
\item By repeated application of M2 we conclude that $f_{a+b\sigma}$ is mono for all $a,b$. 
\end{enumerate}
Accordingly $f_{a+b\sigma}$ is both epi and mono for all $a,b$ and the proof is complete. 
\end{proof}

\subsection{Real Landweber Exactness}\label{sec:Stack}
In this section, we want to prove a version of the Landweber exact functor theorem using the Real bordism spectrum $M\R$. 

The restriction maps $M\R_{k\rho} \to MU_{2k}$ are isomorphisms by \cite[Theorem 2.28]{H-K01}. This defines a graded ring morphism $MU_{2*} \to M\R_\bigstar$ along the morphism
\[2\Z \to RO(C_2),\qquad 2k \mapsto k\rho\]
of the monoids indexing the grading. In particular, $M\R_\bigstar$ becomes a graded $MU_{2*}$-module in a suitable sense.

\begin{defi}
 Let $E\R$ be a strongly even $C_2$-spectrum with underlying spectrum $E$. Then $E\R$ is called \emph{Real Landweber exact} if for every Real orientation $M\R \to E\R$ the induced map 
 \[
M\R_\bigstar(X) \tensor_{MU_{2*}} E_{2*} \to E\R_\bigstar(X)
\]
is an isomorphism for every $C_2$-spectrum $X$.

Here, the gradings can be parsed in the following way: For every $k\in \Z$, we have a $2\Z$-graded $MU_{2*}$-module $M\R_{k+*\rho}(X)$ in the way described above so that the expression $M\R_{k+*\rho}(X)\tensor_{MU_{2*}}E_{2*}$ makes sense in the world of $2\Z=\rho\Z$-graded $MU_{2*}$-modules. Now observe that $RO(C_2)$ is a free abelian group generated by $\rho$ and $1$; thus an $RO(C_2)$-graded abelian group is an equivalent datum to a $\Z$-graded $\Z\rho$-graded abelian group and this expresses what $M\R_\bigstar(X) \tensor_{MU_{2*}} E_{2*}$ means.
\end{defi}

\begin{thm}[Real Landweber exact functor theorem]\label{thm:RealLandweber}\mbox{}
\begin{itemize}
\item[(a)]Let $E_{2*}$ be a graded Landweber exact $MU_{2*}$-algebra, concentrated in even degrees. Then
\[
X\mapsto M\R_\bigstar(X) \tensor_{MU_{2*}} E_{2*}
\]
is a $C_2$-equivariant homology theory.
\item[(b)]Let $E\R$ be a strongly even $C_2$-spectrum whose underlying spectrum $E$ is Landweber exact. Then $E\R$ is Real Landweber exact. 
\end{itemize}
\end{thm}

Let us shortly recall how Landweber exactness is treated non-equivariantly from the stacky point of view. Good sources are, for example, \cite{Goe09}, Chapter 4 of \cite{TMF} or Lectures 11 and 15 of \cite{Lur10}. 

The stack associated to the graded Hopf algebroid $(MU_{2*}, MU_{2*}MU)$ is $\MM_{FG}$, the moduli stack of formal groups. This implies that the category of quasi-coherent sheaves on $\MM_{FG}$ is equivalent to that of evenly graded $(MU_{2*}, MU_{2*}MU)$-comodules (see for example \cite[Remark 34]{Nau07}). The graded comodule $MU_{*+2}$ corresponds to a line bundle $\omega$ on $\MM_{FG}$. This allows us to define the \emph{graded global sections} $\Gamma_{2*}(\FF)$ of a quasi-coherent sheaf $\FF$ on $\MM_{FG}$ as $\Gamma(\FF \tensor \omega^{\tensor *})$. Likewise, the category of quasi-coherent sheaves on $(\Spec\, E_{2*})/\G_m$ is equivalent to that of evenly graded modules over $E_{2*}$; more precisely, a quasi-coherent sheaf $\FF$ on $(\Spec\, E_{2*})/\G_m$ corresponds to the graded module $\Gamma_{2*}(\FF) = \Gamma(\FF\tensor \omega_E^{\tensor *})$, where $\omega_E$ corresponds to the graded module $E_{2*+2}$. We remind the reader here that $(\Spec\, E_{2*})/\G_m$ denotes (as always) the stack quotient. 

An $MU_{2*}$-algebra $E_{2*}$ is Landweber exact iff the composite 
$$f\colon \Spec E_{2*} / \G_m \to \Spec MU_{2*}/\G_m \to \MM_{FG}$$
is flat (if the Landweber exactness criterion is phrased classically using the $v_i$, this is the non-formal part of the proof). Given a spectrum $X$, define quasi-coherent sheaves $\FF^X_i$ for $i=0,1$ on $\MM_{FG}$ corresponding to the graded $(MU_{2*}, MU_{2*}MU)$-comodules $MU_{2*+i}X$. These are functors in $X$ and define ungraded homology theories on spectra with values in quasi-coherent sheaves on $\MM_{FG}$. Because $f$ is flat and thus $f^*$ is exact, the functors
$$X \mapsto f^*\FF_i^X$$
define homology theories with values in quasi-coherent sheaves on $(\Spec\, E_{2*})/\G_m$. We want to identify $\Gamma_{2*}(f^*\FF_i^X)$ with $MU_{2*+ i}(X)\tensor_{MU_{2*}} E_{2*}$. The following lemma provides this identification and thus completes the proof of non-equivariant Landweber exactness. 

\begin{lemma}\label{lem:SheafIdentities}
Let $\FF$ be a quasi-coherent sheaf on $\MM_{FG}$. Then
\begin{align*}
\FF(\Spec MU_{2*}) \tensor_{MU_{2*}} E_{2*} &\cong \Gamma_{2*}((Spec\, E_{2*})/\G_m; f^*\FF)
\end{align*}
where we view $\FF(\Spec MU_{2*})$ as an evenly graded $MU_{2*}$-module. These isomorphisms are natural in $\FF$.
\end{lemma}
\begin{proof}
We have a commutative diagram
\[
\xymatrix{& \Spec MU_{2*}/\G_m \ar[d]^q\\
\Spec E_{2*}/\G_m \ar[ur]^{g} \ar[r]^-f & \MM_{FG}
}
\]

By definition, $q^*\FF$ corresponds to the evenly graded $MU_{2*}$-module $\FF(\Spec MU_{2*})$. Thus, $f^*\FF \cong g^*q^*\FF$ corresponds to the evenly graded $E_{2*}$-module $\FF(\Spec MU_{2*})\tensor_{MU_{2*}} E_{2*}$, proving the lemma.
\end{proof}

Now we turn to the proof of Theorem \ref{thm:RealLandweber}. Part (a) of it can be proven analogously to Landweber exactness in the motivic setting as in \cite{NSO09}, though we follow their approach only loosely. 
The crucial fact about $M\R$ is the following lemma, which was already implicitly treated in \cite{H-K01} and can also be found in \cite{HHR09}.
\begin{lemma}
The restriction
$$(M\R_{*\rho}, M\R_{*\rho}M\R) \to (MU_{2*}, MU_{2*}MU)$$
defines an isomorphism of Hopf algebroids. 
\end{lemma}
\begin{proof}
It is clear that restriction defines a morphism of Hopf algebroids. It is left to show that $M\R_{*\rho}M\R \to MU_{2*}MU$ is an isomorphism. 

Let $\CC$ be the class of all pointed $C_2$-spaces and $\CC^{st}$ the class of all (genuine) $C_2$-spectra $X$ such that
$$M\R_{*\rho}(X) \to MU_{2*}(X)$$
is an isomorphism and
$$M\R_{*\rho-1}(X) \to MU_{2*-1}(X)$$
is a monomorphism, where homology is understood to be reduced in the unstable case. Observe first that $S^0\in\CC$ and $X\in \CC$ if and only if $\Sigma^\infty X \in \CC^{st}$. We have furthermore the following closure properties.
\begin{itemize}
 \item Both $\CC$ and $\CC^{st}$ are closed under weak equivalences and filtered homotopy colimits.
 \item If $X \in\CC^{st}$ and $S^{k\rho-1} \to X$ is a map, then also its cofiber is in $\CC^{st}$ as follows by the five lemma and from $M\R$ being strongly even. 
 \item If $X \in \CC$ and $V\to X$ is a Real vector bundle, then also the Thom space $X^V$ is in $\CC$ as $M\R$ is Real-oriented.
 \item If $X \in \CC^{st}$, then also $\Sigma^{k\rho}X \in \CC^{st}$ for every $k\in\Z$.
\end{itemize}
 We will demonstrate that these properties imply that $M\R \in \CC^{st}$. 

Depending on the model of $M\R$ of choice it is either easy to see or a theorem (\cite[B.252]{HHR09}) that we can write $M\R$ as a directed homotopy colimit over $\Sigma^{-n\rho} MU(n)$, where $MU(n)$ is the suspension spectrum of the Thom space $BU(n)^{\gamma_n}_+$ with the $C_2$-action by complex conjugation (which gives the universal bundle $\gamma_n$ the structure of a Real bundle). The Grassmannian $BU(n)$ is a directed homotopy colimit of finite dimensional Grassmannians, which are built of cells of dimension $k\rho = k\C$ by the theory of Schubert cells. Thus, $M\R$ is in $\CC^{st}$. 
\end{proof}
\begin{proof}[Proof of Theorem \ref{thm:RealLandweber}:] Given a $C_2$-spectrum $X$, define quasi-coherent sheaves $\FF^X_i$ for $i\in\Z$ on $\MM_{FG}$ corresponding to the graded $(MU_{2*}, MU_{2*}MU)\cong (M\R_{*\rho}, M\R_{*\rho}M\R)$-comodules $M\R_{*\rho+i}X$. As above, the $\FF^X_i$ are $C_2$-equivariant homology theories with values in quasi-coherent sheaves on $\MM_{FG}$. Thus, the pullbacks $f^*\FF_i^X$ are homology theories with values in quasi-coherent sheaves on $\Spec E_{2*}/\G_m$. By Lemma \ref{lem:SheafIdentities}, the associated graded module is $M\R_{*\rho + i}(X)\tensor_{MU_{2*}} E_{2*}$, which is thus a homology theory as well; this proves part (a) of Real Landweber exactness. Note that as $M\R_{*\rho + i}(X)\tensor_{MU_{2*}} E_{2*}$ has suspension isomorphisms for arbitrary (virtual) representations, it is isomorphic to the $RO(C_2)$-graded theory associated to its degree-$0$ part.      

For the proof of (b), choose a Real orientation $M\R \to E\R$, which exists by Lemma \ref{lem:RealOrientable}. By Lemma \ref{lem:regrep} it is now enough to show that the induced maps
\[
M\R_{*\rho} \tensor_{MU_{2*}} E_{2*} \to E\R_{*\rho}
\]
and
\[
MU_{2*} \tensor_{MU_{2*}} E_{2*} \to E_{2*}
\]
are isomorphisms (as the odd groups are zero anyhow). The latter is clear and the former is true since both $\mpi_{*\rho}M\R$ and $\mpi_{*\rho}E\R$ are constant.
\end{proof}

By the following proposition, the Real Landweber exact functor theorem can actually be used to produce $C_2$-spectra. 
\begin{thm}\label{thm:Brown}
 Any (ungraded) $G$-equivariant homology theory can be represented by a $G$-spectrum, i.e.\ for every $G$-equivariant homology theory $h_0$, there is a $G$-spectrum $E$ such that there are isomorphisms $\pi_0^G(X \sm E) \cong h_0(X)$, natural in a $G$-spectra $X$. Note that this implies natural isomorphisms $\pi^G_{\bigstar}(X\sm E) \cong h_{\bigstar}(X)$ as well. 
 
 Moreover, any transformation of $G$-equivariant homology theories can be represented by a map of $G$-spectra. 
\end{thm}
\begin{proof}
 By \cite[Corollary 9.4.4]{HPS97}, the homotopy category of genuine $G$-spectra is a Brown category, which means exactly the statement of our proposition. 
\end{proof}
                                                                                                                                                                                                                                                                                                                                                                                                                                                                                                                                                                                                                                                              
In the rest of the section, we will give some reformulations of the stacky point of view on Landweber exactness to show that two Real Landweber exact spectra are equivalent iff their underlying spectra are equivalent. The following easy lemma will be useful. 

\begin{lemma}\label{lem:projectionformula}
 Let $f\colon \XX \to \YY$ be an affine morphism of algebraic stacks (in the sense of \cite[Definition 6]{Nau07}) and $\FF$ be a quasi-coherent sheaf on $\XX$ and $\GG$ be a quasi-coherent sheaf on $\YY$. Then the canonical homomorphism
 \[ 
  f_*\FF \tensor_{\OO_{\YY}} \GG \to  f_*(\FF\tensor_{\OO_{\XX}} f^*\GG)
 \]
 is an isomorphism.
\end{lemma}
\begin{proof}
 We can assume that $\YY$ is affine and hence also $\XX$. Then it is clear. 
\end{proof}

Recall the notation $\FF_i^X$ from the proof of (Real) Landweber exactness above. 

\begin{prop}\label{prop:Landstack}\mbox{}
\begin{itemize}
\item[(a)]Let $E$ be an even Landweber exact spectrum. The associated graded formal group on $E_{2*}$ defines a map 
$$f\colon (\Spec E_{2*})/\G_m \to (\Spec MU_{2*})/\G_m \to \MM_{FG}.$$ 
Then given a spectrum $X$, we have
\[E_*(X) \cong \Gamma_{2*}(\MM_{FG};\FF_*^X\tensor_{\OO_{\MM_{FG}}}f_*\OO_{(\Spec E_{2*})/\G_m}).\]
\item[(b)]Let $E\R$ be an even Real Landweber exact spectrum. The associated graded formal group on $E\R_{*\rho} \cong E_{2*}$ (for $E$ the underlying spectrum) defines a map $$f\colon (\Spec E_{2*}) /\G_m \to \MM_{FG}$$
as above. Then given a $C_2$-spectrum $X$, we have
\[E\R_\bigstar(X) \cong \Gamma_{2*}(\MM_{FG};\FF_*^X\tensor_{\OO_{\MM_{FG}}}f_*\OO_{(\Spec E_{2*})/\G_m}),\]
 where $\Gamma_{2n}(\MM_{FG}; \FF_i^X \tensor \cdots)$ is in degree $n\rho+i$. 
\end{itemize}
\end{prop}
\begin{proof}
 We will prove only part (a); the proof of part (b) only needs change in notation. As in the proof of Landweber exactness, the left hand side decomposes into two pieces of the form $\Gamma_{2*}(f^*\FF_i^X) \cong \Gamma_{2*}(f_*f^*\FF_i^X).$
 Thus, we only have to show that 
 $$f_*f^*\FF_i^X \cong \FF_i^X \tensor_{\OO_{\MM_{FG}}} f_*\OO_{(\Spec E_{2*})/\G_m},$$
 which follows directly from Lemma \ref{lem:projectionformula} with $\FF = \OO_{(\Spec E_{2*})/\G_m}$ and $\GG = \FF_i^X$. 
\end{proof}

In particular, we see that the values of a (Real) Landweber exact theory do not depend on the $MU_{2*}$-module structure of $E_{2*}$, but only on the graded quasi-coherent sheaf $f_*\OO_{(\Spec E_{2*})/\G_m}$ on $\MM_{FG}$ defined by $E_{2*}$. This sheaf has an alternative description:

\begin{lemma}\label{lem:SheafIdentification}
Let $E$ be an even Landweber exact spectrum and $f\colon \Spec E_{2*}/\G_m \to \MM_{FG}$ as above. Then we have an isomorphism $f_*\OO_{(\Spec E_{2*})/\G_m} \cong \FF^E_0$.
\end{lemma}
\begin{proof}
This was proven in the even-periodic context in the proof of \cite[Proposition 2.4]{MM15}. The general case is similar.
\end{proof}

\begin{prop}\label{prop:RealUnderlying}
 Let $E\R$ and $F\R$ be two (strongly even) Real Landweber exact $C_2$-spectra, whose underlying spectra $E$ and $F$ are equivalent. Then $E\R$ and $F\R$ are equivalent. 
\end{prop}
\begin{proof}
 Assume $E\simeq F$. Then they define isomorphic graded quasi-coherent sheaves $\FF^E_0$ and $\FF^F_0$ on $\MM_{FG}$. As $E\R$ and $F\R$ are Real Landweber exact, Proposition \ref{prop:Landstack} and Lemma \ref{lem:SheafIdentification} imply the following chain of isomorphisms, natural in a $C_2$-spectrum $X$:
\begin{align*} E\R_0(X) &\cong \Gamma(\MM_{FG}; \FF^X_0 \tensor_{\OO_{\MM_{FG}}} \FF^E_0) \\
 &\cong \Gamma(\MM_{FG}; \FF^X_0 \tensor_{\OO_{\MM_{FG}}} \FF^F_0) \\
 &\cong F\R_0(X)
\end{align*}
Thus, the (ungraded) $C_2$-equivariant homology theories defined by $E\R$ and by $F\R$ are isomorphic. By Proposition \ref{thm:Brown}, a natural isomorphism of $C_2$-equivariant homology theories induces an equivalence of the representing $C_2$-spectra.
\end{proof}

\subsection{Forms of \texorpdfstring{$BP\R\langle n\rangle$}{BPRn} and \texorpdfstring{$E\R(n)$}{ERn}}
Fix a prime $p$.
\begin{defi}
Let $E$ be a complex oriented $p$-local commutative and associative ring spectrum (up to homotopy). The $p$-typification of its formal group law defines a ring morphism $BP_*\to E_*$.
\begin{itemize}
\item[(a)]We call $E$ a \emph{form of $BP\langle n\rangle$} if the map
\[\Z_{(p)}[v_1,\dots, v_n] \subset BP_* \to E_*\]
is an isomorphism. This does not depend on the choice of $v_i$.
\item[(b)]We call $E$ a \emph{form of $E(n)$} if there is a choice of indecomposables $v_1,\dots, v_n \in BP_*$ with $|v_i| = 2(p^i-1)$ such that the image of $v_n$ under the homomorphism
\[\mathbb{Z}_{(p)}[v_1,\dots, v_n] \subset BP_* \to E_*\]
is invertible and the induced morphism $\mathbb{Z}_{(p)}[v_1,\dots, v_n, v_n^{-1}] \to E_*$ is an isomorphism. 
\end{itemize}
\end{defi}

Spectra like in (a) are also sometimes called \emph{generalized} $BP\langle n\rangle$ (see \cite[Def 4.1]{L-N14}). There is a Real analogue, where we specialize to $p=2$:

\begin{defi}
Let $E\R$ be an even Real oriented $2$-local commutative and associative $C_2$-ring spectrum (up to homotopy). This induces a formal group law on $E\R_{
\rho}$ \cite[Thm 2.10]{H-K01}; its $2$-typification defines a map $BP_{2*} \cong BP\R_{*\rho} \to E\R_{*\rho}$.
\begin{itemize}
\item[(a)]We call $E\R$ a \emph{form of $BP\R\langle n\rangle$} if the map
\[
\m{\Z}_{(2)}[\vb_1,\dots, \vb_n] \subset \mpi_{*\rho}BP\R \to \mpi_{*\rho} E\R
\]
is an isomorphism of constant Mackey functors. This does not depend on the choice of $\vb_i$.
\item[(b)]We call $E\R$ a \emph{form of $E\R(n)$} if there is a choice of indecomposables $\vb_1,\dots, \vb_n \in BP\R_{*\rho}$ with $|\vb_i| = (2^i-1)\rho$ such that the image of $\vb_n$ under the homomorphism
\[\mathbb{Z}_{(2)}[\vb_1,\dots, \vb_n] \subset BP\R_{*\rho} \to E\R_{*\rho}\]
is invertible and the induced morphism $\m{\mathbb{Z}}_{(2)}[\vb_1,\dots, \vb_n, \vb_n^{-1}] \to \mpi_{*\rho}E\R$ is an isomorphism of constant Mackey functors.
\end{itemize}
\end{defi}

Note that a form $E\R(n)$ is always Real Landweber exact by Theorem \ref{thm:RealLandweber} as it is strongly even and its underlying spectrum is Landweber exact.

\begin{prop}\label{prop:JWUniqueness}
If for two forms of $E\R(n)$ their underlying spectra are equivalent, then they are equivalent as $C_2$-spectra.
\end{prop}
\begin{proof}
As every form of $E\R(n)$ is Real Landweber exact, this follows directly from Proposition \ref{prop:RealUnderlying}.
\end{proof}

\section{\texorpdfstring{$TMF_1(3)$}{TMF13} and friends}\label{sec:TMFandFriends}
In this section, we will first define the versions of $TMF$ we are after and compute $\pi_*Tmf_1(3)$. In Subsection \ref{sec:specseq}, we will run the homotopy fixed point spectral sequence for $tmf_1(3)^{hC_2}$ and apply this to see that $tmf_1(3)$ is a form of $BP\R\langle 2\rangle$. In subsection \ref{sec:relationship}, we will discuss the relationship between the $C_2$-spectra $tmf_1(3)$, $Tmf_1(3)$ and $TMF_1(3)$. In particular, we will show that $TMF_1(3) \simeq tmf_1(3)[\overline{\Delta}^{-1}]$ and how this implies the Real Landweber exactness of $TMF_1(3)$. 
\subsection{Basics}\label{sec:Basics}
Denote by $\MM_{ell}$ the moduli stack of elliptic curves and by $\MMb_{ell}$ its compactification.
Mapping an elliptic curve to its formal group defines a flat map $\MMb_{ell} \to \MM_{FG}$ to the moduli stack of formal groups. By \cite{HL13} (extending earlier work by Goerss, Hopkins and Miller), the induced presheaf of even-periodic Landweber exact homology theories refines to a sheaf of $E_\infty$-ring spectra $\OO^{top}$ on the log-\'etale site of $\MMb_{ell}$.

Denote by $\MM_1(n)$ the moduli stack of elliptic curves with one chosen point of exact order $n$ and by $\MMb_1(n)$ its compactification, whose definition we will now review. In the compactification we have to allow not only (smooth) elliptic curves, but \emph{generalized elliptic curves}, which can have as fibers also N{\'e}ron $m$-gons for $m|n$. These are obtained by gluing $m$ copies of $\mathbb{P}^1$, where
$0$ in the $i$th $\mathbb{P}^1$ (for $i \in \mathbb{Z}/m\mathbb{Z}$) is
attached to $\infty$ in the $(i+1)$st. For precise definitions see \cite[Section II.1]{D-R73}.
\begin{definition}[\cite{D-R73}, IV.4.11-4.15, \cite{Con07}]
We define the stack $\MMb_1(n)$ to classify generalized elliptic curves $p\co \EE \to S$ over
a base $S$ with $n$ invertible, together with an injection of group schemes $\mathbb{Z}/n \mathbb{Z} \to
\EE^{\circ}$ from the constant group scheme $\Z/n\Z$ over $S$ into the smooth locus of $\EE$ such that
\begin{enumerate}
 \item each  geometric fiber $\Spec \overline{k}\times_S \EE$ of $p$ is either smooth or a N{\'e}ron $m$-gon for some $m|n$
 \item the image of $\mathbb{Z}/n\mathbb{Z}$ intersects each irreducible component in every geometric fiber of $\EE$ nontrivially
\end{enumerate}
\end{definition}

We define
\begin{align*}
TMF_1(n) &= \OO^{top}(\MM_1(n))\\
Tmf_1(n) &= \OO^{top}(\MMb_1(n)) \\
tmf_1(n) &= \tau_{\geq 0}Tmf_1(n)
\end{align*}

We remark that the last definition should only be considered appropriate for $n\geq 2$ if $tmf_1(n)$ is even and $\pi_{2n}tmf_1(n)$ is isomorphic to the ring of integral holomorphic modular forms $H^0(\MMb_1(n);\omega^{\tensor *}).$
The second assumption is always fulfilled, but in general there can be a non-trivial $\pi_1tmf_1(n)$, which is isomorphic to $H^1(\MMb_1(n);\omega)$ (as already remarked in \cite[Remark 6.4]{HL13}). Luckily, there are no such problems for $tmf_1(3)$ as we will see at the end of this subsection. 

The following lemma is well-known:
\begin{lemma}\label{lem:Landweber}
 The spectrum $TMF_1(n)$ is Landweber exact for $n\geq 2$.
\end{lemma}
\begin{proof}
Throughout the proof, we will use the notations $\omega$ and $\FF_i^X$ (for $i = 0,1$ and $X$ a spectrum) from Section \ref{sec:Stack}.

First, we prove that for $\Spec A \to \MM_{ell}$ \'etale such that the pullback of $\omega$ to $\Spec A$ is trivial, $E = \OO^{top}(\Spec A)$ is Landweber exact: By the descent spectral sequence, $E$ is even periodic. Thus, we can choose a complex orientation $MU \to E$. This defines a formal group law on $A = \pi_0E$. By construction (see Behrens's article in \cite{TMF}), the composite morphism $g\colon\Spec A \to \MM_{ell} \to \MM_{FG}$ classifies the underlying formal group. As $g$ is flat, a version of the Landweber exact functor theorem (see Lemma \ref{lem:SheafIdentities} or \cite[Lecture 15]{Lur10}) implies that the source of
\[(X\mapsto MU_*(X)\tensor_{MU_*}E_*) \to (X\mapsto E_*(X))\]
is a homology theory and thus the depicted morphism is a natural isomorphism. This proves that $E$ is Landweber exact. Furthermore, it provides a natural isomorphism between $\pi_{2k-i}(\OO^{top}\sm X)$ and the pullback of $\FF_i^X\tensor \omega^{\tensor k}$ to $\MM_{ell}$ for $i=0,1$.

 For $n\geq 4$, the stack $\MM_1(n)$ is represented by an affine scheme (\cite[2.7.3 and 4.7.0]{K-M85}). For $n=2,3$ we have the slightly weaker statement that only $\MM_1^1(n)$ is of the form $\Spec A$, where $\MM_1^1(n)$ classifies elliptic curves where we choose not only a point of order $n$, but also a nowhere vanishing invariant differential; we recover $\MM_1(n)$ as $\Spec A/\G_m$. This can either be shown along the same lines as the previous statement or deduced from concrete presentations (see e.g.\ \cite[Proposition 3.2]{M-R09} and \cite[Section 1.3]{Beh06}). In particular, the global sections functor
 \[\Gamma\colon \QCoh(\MM_1(n)) \to \mathrm{Abelian Groups}\]
 on quasi-coherent sheaves is exact. Indeed, $\QCoh(\Spec A/\G_m)$ is by Galois descent equivalent to the category of graded $A$-modules (where the grading comes from the $\G_m$-action). The global sections functor corresponds to $M_* \mapsto M_0$, which is clearly exact.

 In particular, we see that the descent spectral sequence
 \[H^s(\MM_1(n); h^*\omega^{\tensor t}) \Rightarrow \pi_{2t-s}TMF_1(n)\]
 is concentrated in the $0$-line, where $h\colon \MM_1(n) \to \MM_{FG}$ classifies the formal group. Thus, $\MM_1(n) \simeq (\Spec \pi_{2*}TMF_1(n))/\G_m$. By the same argument we get a natural isomorphism
 \[TMF_1(n)_{2k-i}(X) \cong H^0(\MM_1(n); h^*(\FF_i^X\tensor \omega^{\tensor k}))\] 
 for spectra $X$.
 Now Lemma \ref{lem:SheafIdentities} implies that we have isomorphisms
 \[MU_*(X) \tensor_{MU_*} TMF_1(n) \cong TMF_1(n)_*(X),\]
 again natural in $X$.
\end{proof}

Sending the point $x$ of order $n$ to $[k]x$ for $k\in (\Z/n)^\times$, defines a $(\Z/n)^\times$-action on $\MMb_1(n)$. In particular, this induces $(\Z/3)^\times = C_2$-actions on $TMF_1(3)$ and $Tmf_1(3)$. Thus we will view these spectra as cofree $C_2$-spectra as in Section \ref{sec:naive}. We define the $C_2$-spectrum $tmf_1(3)$ as the $C_2$-equivariant connective cover of $Tmf_1(3)$ so that 
$$tmf_1(3)^{C_2} = \tau_{\geq 0} \left(Tmf_1(3)^{hC_2}\right).$$ 
Note that this spectrum is not cofree as
\[\tau_{\geq 0} \left(Tmf_1(3)^{hC_2}\right)\simeq \tau_{\geq 0} \left(tmf_1(3)^{hC_2}\right),\]
which follows formally from the $C_2$-homotopy fixed point spectral sequence, and $tmf_1(3)^{hC_2}$ has negative homotopy groups as we will see in the next subsection. 

Denote by $\MM_0(n)$ the moduli stack of elliptic curves with a chosen subgroup of order $n$ and by $\MMb_0(n)$ its compactification, defined as follows:
\begin{defi}We define $\MMb_0(n)$ for $n$ squarefree\footnote{For the subtleties for non-squarefree $n$ see \cite{Ces15}.} to classify generalized elliptic curves $p\co \EE \to S$ over a base
$S$ with $n$ invertible, together with a subgroup $G \subset \EE^{\circ}[n]$ such that
\begin{enumerate}
 \item each  geometric fiber of $p$ is either smooth or a N{\'e}ron $m$-gon for some $m|n$
 \item $G$ is \'etale locally isomorphic to $\Z/n\Z$
 \item $G$ intersects each irreducible component in every geometric fiber of $\EE$ nontrivially
\end{enumerate}
\end{defi}
We define
\begin{align*}
TMF_0(3) &= \OO^{top}(\MM_0(3))\\
Tmf_0(3) &= \OO^{top}(\MMb_0(3))
\end{align*}
The forgetful maps
\[\MM_1(n) \to \MM_0(n) \quad \text{and} \quad \MMb_1(n) \to \MMb_0(n)\]
are $(\Z/n)^\times$-Galois coverings for $n$ squarefree, as checked in \cite[Theorem 7.12]{MM15}. In particular, this implies that $Tmf_0(3) \simeq Tmf_1(3)^{hC_2}$ and $TMF_0(3) \simeq Tmf_1(3)^{hC_2}$. If we define $tmf_0(3) = \tau_{\geq 0}Tmf_0(3)$, then it follows that $tmf_0(3) \simeq tmf_1(3)^{C_2}$.\\

Next, we will study $\MMb_1(3)$ in more detail. The following lemma essentially says that there can be only one reasonable compactification of our moduli stacks.
\begin{lemma}\label{lem:Normalization}
Let $f\co Y\to X$ be a map of Deligne--Mumford stacks (over some base scheme $S$) and assume that $X$ is locally noetherian. Let $\overline{f}_1, \overline{f}_2\co \overline{Y}_1, \overline{Y}_2 \to X$ be finite morphisms from normal Deligne--Mumford stacks (over $S$) such that $Y$ sits inside $\overline{Y}_1$ and  $\overline{Y}_2$ as a dense open substack and $\overline{f}_i|_Y = f$ for $i=1,2$. Then $\overline{Y}_1 \simeq \overline{Y}_2$ as stacks over $X$.
\end{lemma}
\begin{proof}
The same argument as in \cite[Cor 12.2]{G-W10} shows that it is enough to show that $(\overline{f}_1)_*\OO_{\overline{Y}_1} \cong (\overline{f}_2)_*\OO_{\overline{Y}_2}$ as both $\overline{f}_1$ and $\overline{f}_2$ are affine. In particular, it is enough to show the existence of a natural isomorphism between $\overline{Y}_1$ and $\overline{Y}_2$ if $X = \Spec A$ is affine. Then also $\overline{Y}_1 = \Spec B_1$ and  $\overline{Y}_2 = \Spec B_2$ are affine. The inclusions $Y\subset \Spec B_1$ and $Y\subset \Spec B_2$ induce bijections of the sets of connected components. As all these schemes are normal and locally noetherian, all connected components are irreducible \cite[Rem 6.3.7]{G-W10}. Thus, we can assume that $\overline{Y}_1$ and $\overline{Y}_2$ are irreducible and hence $B_1$ and $B_2$ are normal integral domains. Write $C = \Gamma(\OO_Y)$. As $Y$ is open in $\overline{Y}_i$, the map from $B_i$ into its fraction field factors over $C$. In particular, $B_i$ injects into $C$ and is integrally closed in 
it. As it is also finite and thus integral over $A$, it consists exactly of those elements
in $C$ that are integral over
$A$. In particular, we have a canonical isomorphism $B_1\cong B_2$ of $A$-algebras.
\end{proof}

Note that $\MMb_1(n) \to \MMb_{ell}[\frac1n]$ and $\MMb_0(n)\to \MMb_{ell}[\frac1n]$ (if $n$ is squarefree) are finite morphisms from normal (even regular) Deligne--Mumford stacks and $\MM_1(n)\subset \MMb_1(n)$ and $\MM_0(n) \subset \MMb_0(n)$ are open dense inclusions (see \cite[IV.3.4]{D-R73} or \cite[4.1.1]{Con07}; note that the complement of an effective Cartier divisor is open and dense). Thus, we can apply the previous lemma to approach the following well-known result (see e.g.\ \cite{L-N14}) that has to the knowledge of the authors not appeared with full proof in print.

\begin{prop}\label{prop:StackIdentification}
We have equivalences
\begin{align*}
\MM_1(3) &\simeq \Spec \big(\Z\big[\tfrac13\big][a_1,a_3][\Delta^{-1}]\big)/\G_m \\
\MMb_1(3) &\simeq \Big(\Spec \big(\Z\big[\tfrac13\big][a_1,a_3]\big)\setminus \{0\}\Big)/\G_m =: \PP_{\Z\big[\tfrac13\big]}(1,3)
\end{align*}
Here,
\begin{itemize}
\item the $\G_m$-action on $\Spec \big(\Z\big[\tfrac13\big][a_1,a_3]\big)$ is induced by the grading with $|a_1| = 1$ and $|a_3| = 3$,
\item $\Delta = a_3^3(a_1^3-27a_3)$,
\item $\{0\}$ denotes the common vanishing locus of $a_1$ and $a_3$
\item $\PP_{\Z\big[\tfrac13\big]}(1,3)$ is often called the weighted (stacky) projective line with weights $1$ and $3$.
\end{itemize}
\end{prop}
\begin{proof}
The first equivalence follows from \cite[Proposition 3.2]{M-R09}.

Set $A = \Z\big[\tfrac13\big][a_1,a_3]$. The equality $\PP_{\Z\big[\tfrac13\big]}(1,3) = (\Spec A\setminus \{0\})/\G_m$ is just the definition of the weighted projective line. This is a proper and smooth Deligne--Mumford stack over $\Spec \Z\big[\tfrac13\big]$ by \cite[Proposition 2.1, Remark 2.2]{Mei15}. Note furthermore that $\MM_1(3) \subset \PP_{\Z\big[\tfrac13\big]}(1,3)$ is a dense open substack.

To apply Lemma \ref{lem:Normalization}, we need to construct a finite morphism
\[\PP_{\Z\big[\tfrac13\big]}(1,3) \to \MMb\left[\tfrac13\right]\]
 that extends the morphism
 \[\MM_1(3) \to \MM_{ell}\left[\tfrac13\right] \subset \MMb_{ell}\left[\tfrac13\right].\]

The equation $y^2+a_1xy+a_3y = x^3$ defines a cubic curve over $\Spec A/\G_m$. We want to show that this equation actually defines a generalized elliptic curve $E$ over $\PP_{\Z\big[\tfrac13\big]}(1,3)$. For this, we have to check that for no map $f\co \Spec k \to \PP_{\Z\big[\tfrac13\big]}(1,3)$ for $k$ a field (of characteristic $\neq 3$), the pullback $f^*E$ has a cusp. Equivalently, we have to show that for any values $a_1,a_3\in k$ for which $c_4 = a_1^4-24a_1a_3$ and $\Delta = a_3^3(a_1^3-27a_3)$ vanish, also $a_1$ and $a_3$ vanish. First observe that if $c_4 = \Delta = 0$, then $a_1 = 0$ implies $a_3 = 0$ and vice versa. If $\Delta = 0$, either $a_3= 0$ or $a_1^3 = 27a_3$. In the second case, $27a_1a_3 = a_1^4 = 24a_1a_3$ and thus $a_1= 0$ or $a_3=0$.

Thus, we obtain a map $p\colon \Spec A/\G_m \to \MM_{cub}\left[\tfrac13\right]$ to the moduli stack of cubic curves that restricts to a map $\PP_{\Z\big[\tfrac13\big]}(1,3) \to \MMb_{ell}\left[\tfrac13\right]$, which in turn extends the map $\MM_1(3) \to \MM_{ell}\left[\tfrac13\right]\subset \MMb_{ell}\left[\tfrac13\right]$.

As computed in the beginning of Section 7 of \cite{Bau08}, the map $p$ is surjective and we have $\Spec A/\G_m \times_{\MM_{cub}} \Spec A/\G_m \simeq (\Spec A[s,t]/(f,g))/\G_m$, where $f$ and $g$ are polynomials in $s$ and $t$ such that $A[s,t]/(f,g)$ is a finite flat $A$-module. As finiteness can be checked after fpqc-base change, the map $p$ is finite and hence also its restriction $\PP_{\Z\big[\tfrac13\big]}(1,3) \to \MMb_{ell}\left[\tfrac13\right]$, which is the base change $p\times_{\MM_{cub}\left[\tfrac13\right]}\MMb_{ell}\left[\tfrac13\right]$. Thus, the result follows by Lemma \ref{lem:Normalization}.
\end{proof}

By checking the gradings, we see that $p^*\omega\cong \OO(1)$ for $p\co \PP_{\Z\big[\tfrac13\big]}(1,3) \to \MMb_{ell}\left[\tfrac13\right]$ the restriction of the morphism constructed in the proof above. (Here, $\omega$ denotes the line bundle $\pi_2\OO^{top}$ on $\MMb_{ell}$, which is also the pullback of the line bundle on $\MM_{FG}$ we have denoted before by $\omega$.) Thus, we have
\[H^s(\MMb_1(3); \omega^{\tensor \ast}) \cong
\begin{cases} \Z\big[\tfrac13\big][a_1, a_3] & \text{ for } s=0 \\
\Z\big[\tfrac13\big][a_1, a_3]/(a_1^\infty, a_3^\infty) & \text { for } s=1\\
0 & \text { for } s \geq 2 \end{cases}\]
as shown, for example, in \cite[Proposition 2.5]{Mei15}. Here, $\Z\big[\tfrac13\big][a_1, a_3]/(a_1^\infty, a_3^\infty)$ denotes the $\Z\big[\tfrac13\big][a_1, a_3]$-torsion module with $\Z\big[\tfrac13\big]$-basis given by the monomials $\frac1{a_1^ia_3^j}$, $i,j\geq 1$. Thus, the descent spectral sequence for $Tmf_1(3)$ collapses. In particular, we see that $\pi_*tmf_1(3) = \Z\big[\tfrac13\big][a_1,a_3]$.

\subsection{\texorpdfstring{$RO(C_2)$}{ROC2}-graded homotopy of \texorpdfstring{$tmf_1(3)$}{tmf13}}\label{sec:specseq}
Our goal in this subsection is to understand the $C_2$-equivariant $RO(C_2)$-graded homotopy groups of $tmf_1(3)$. We will compute this via an $RO(C_{2})$-graded homotopy fixed point spectral sequence, as described for general $G$ in Section \ref{sec:RO(G)HFPSS}. When $G=C_{2}$ there are two important simplifications. The first allows us to identify the $E_{2}$ term more transparently:

\begin{lemma}
Let $E$ be a $C_2$-spectrum. Then
\[
\pi_*(E\sm S^{\sigma-1}) \cong \pi_*E \tensor \sgn
\]
as $C_2$-modules.
\end{lemma}
\begin{proof}
This follows from the fact that the action map $t\colon S^\sigma \to S^\sigma$ has degree $-1$.
\end{proof}

\begin{corollary}\label{cor:ROLandweber}
If $E$ is a $C_{2}$-spectrum, then the $RO(C_{2})$-graded homotopy fixed point spectral sequence has the form
\[
H^{s}(C_{2};\pi_{t}(E)\otimes \sgn^{\tensor r})\Rightarrow \pi^{C_{2}}_{t-s+(\sigma-1)r} F\big(EC_{2+},E\big).
\]
The differential $d_i$ goes from degree $(r,s,t)$ to $(r, s+i, t+i-1)$. The tridegree $(r,s,t)$ corresponds to the bidegree $((t-r)+r\sigma, s)$ in representation grading. 

If $E$ is even with $\pi_{2n}$ flat over $\Z$ and the group $C_2$ acts on $\pi_{2n}E$ via $(-1)^n$, then the $E^2$-term is isomorphic to
\[\overline{\pi_{2*}E} \tensor \Z[u_{2\sigma}^{\pm 1}, a_\sigma]/2a_\sigma\]
with $|u_{2\sigma}| = (2-2\sigma,0)$ and $|a_\sigma| = (-\sigma, 1)$. 
Here, $\overline{\pi_{2n}E}$ is the group $\pi_{2n}E$, but not in degree $2n$, but in degree $n+n\sigma$. 
\end{corollary}
\begin{proof}
 The first part is clear. For the second, note that the $RO(C_2)$-graded $C_2$-representation $\pi_\bigstar E$ is isomorphic to $\overline{\pi_{2*}E} \tensor \bigoplus_{r\in \Z} \sgn^{\tensor r}$ with $\sgn^{\tensor r}$ in degree $r(1-\sigma)$. The first tensor factor is invariant under the $C_2$-action and can therefore be pulled out of the cohomology group. For the second one, we have $H^*(C_2; \bigoplus_{r\in \Z} \sgn^{\tensor r}) \cong \Z[u^{\pm 1}, a]/2a$ with $u\in H^0(C_2; \sgn^{\tensor 2})$ and $a\in H^1(C_2; \sgn)$. 
\end{proof}

The second $C_{2}$ simplification is a recasting of the $RO(C_{2})$-graded homotopy fixed points spectral sequence in a way that allows us to read off permanent cycles. Recall that there is a $C_{2}$-equivariant map
\[
a_{\sigma}\colon S^{0}\to S^{\sigma}
\]
which is essential but for which the restriction is null. The following is undoubtedly well-known to experts.
\begin{lemma}\label{lem:Bockstein}
The $RO(C_{2})$-graded homotopy fixed points spectral sequence for a $C_{2}$-spectrum $X$ coincides with the $a_{\sigma}$-Bockstein spectral sequence for $X$.
\end{lemma}
\begin{proof}
The map $a_{\sigma}^{n}$ fits in a cofiber sequence
\[
S(n\sigma)_{+}\to S^{0}\xrightarrow{a_{\sigma}^{n}} S^{n\sigma},
\]
where $S(n\sigma)$ is the unit sphere in the representation $n\sigma$. Applying $F(-,X)$, we deduce a cofiber sequence of spectra
\[
\Sigma^{-n\sigma} X\xrightarrow{a_{\sigma}^{n}} X\to F\big(S(n\sigma)_{+},X\big).
\]
The space $S(n\sigma)_{+}$ is also the $(n-1)$-skeleton of the standard model for $EC_{2+}$ as the infinite sign sphere, and the map on function spectra induced by the inclusion of the $(n-1)$-skeleton into the $n$-skeleton coincides with the obvious map of cofibers:
\[
\xymatrix{
{\Sigma^{-(n+1)\sigma} X}\ar[r]^-{a_{\sigma}^{n+1}}\ar[d]_{a_{\sigma}} & {X} \ar[r]\ar[d]_{1} & {F\big(S((n+1)\sigma)_{+},X\big)}\ar[d]  \\
{\Sigma^{-n\sigma} X}\ar[r]_-{a_{\sigma}^{n}} & {X}\ar[r] & {F\big(S(n\sigma)_{+},X\big).}
}
\]
Thus the filtration by powers of $a_{\sigma}$ and the filtration by the skeleton of $EC_{2+}$ coincide.
\end{proof}

Recall now from Section~\ref{sec:Basics} that non-equivariantly $$\pi_*tmf_1(3) \cong \Z\big[\tfrac13\big][a_1, a_3]$$
and $$\pi_*TMF_1(3) \cong \Z\big[\tfrac13\big][a_1, a_3, \Delta^{-1}]$$
with $|a_1| =2$ and $|a_3| = 6$. By \cite[Proposition 3.4]{M-R09}, the group $C_2$ acts by $-1$ on $a_1$ and $a_3$ in $\pi_*TMF_1(3)$ and hence also in $\pi_*tmf_1(3)$, as $\pi_*tmf_1(3)$ sits inside $\pi_*TMF_1(3)$. 

By Corollary~\ref{cor:ROLandweber}, the $RO(C_{2})$-graded homotopy fixed point spectral sequence $E_{2}$ term for $tmf_1(3)^{hC_2}$ can be written as
\begin{align}\label{eq:E2term}
E_{2}^{\ast,\ast}=\Z\big[\tfrac13\big][a_\sigma, u_{2\sigma}^{\pm 1}, \ab_1, \ab_3]/(2a_\sigma)
\end{align}
with degrees 
\begin{align*}
 |a_\sigma| &= (-\sigma, 1) = (1-\rho,1), \\
 |u_{2\sigma}| &= (2-2\sigma, 0) = (4-2\rho,0), \\
 |\ab_1| &= (1+\sigma, 0) = (\rho,0) ,\\
 |\overline{a}_3| &= (3+3\sigma, 0) = (3\rho,0).
\end{align*}

We start by identifying the permanent cycles corresponding to $\eta$ and $\nu$ in the Hurewicz image in $\pi_*tmf_1(3)^{hC_2}$. By \cite[Theorem 6.2]{HL13}, there is a $C_2$-equivariant map $$Tmf_1(3) \to KU$$
of $E_\infty$-ring spectra into K-theory, inducing a map between the homotopy fixed point spectral sequences for $Tmf_1(3)^{hC_2}$ and $KO\simeq KU^{hC_2}$. In the latter, $\eta$ is of filtration $1$, so it has to be of filtration $\leq 1$ in the former. As the homotopy fixed point spectral sequences of $Tmf_1(3)^{hC_2}$ and $tmf_1(3)^{hC_2}$ agree in nonnegative degrees, $\eta$ is also filtration $1$ in the homotopy fixed point spectral sequence for $tmf_1(3)^{hC_2}$ and is thus detected by $a_\sigma \ab_1$.

To identify $\nu$, we observe the following lemma:
\begin{lemma}The composite $Tmf\!\left[\tfrac13\right] \xrightarrow{\res} Tmf_0(3) \xrightarrow{\tr} Tmf\!\left[\tfrac13\right]$ is multiplication by $4$.
\end{lemma}
\begin{proof}
This is true on the level of $E_2$-terms of homotopy fixed point spectral sequences, expressing $Tmf_0(3)$ and $Tmf\!\left[\tfrac13\right]$ as homotopy fixed points of $Tmf(3)$ (as the subgroup of matrices of the form $\begin{pmatrix}a&b\\0&d\end{pmatrix}$ in $GL_2(\Z/3)$ has index $4$). The $Tmf\!\left[\tfrac13\right]$-linear self-maps of $Tmf\!\left[\tfrac13\right]$ are in one-to-one correspondence to elements in $\pi_0Tmf\!\left[\tfrac13\right]$. These are all of filtration $0$ in the descent spectral sequence by \cite[Figure 26]{Kon12} and thus detected by their action on 
\[
\pi_0Tmf\!\left[\tfrac13\right] = H^0\big(GL_2(\Z/3); \pi_0Tmf(3)\big).
\]
(As the arguments in \cite{Kon12} are computationally involved, we also sketch another way to arrive at this last result. If there were contributions of positive filtration to $\pi_0Tmf\!\left[\tfrac13\right]$ in the descent spectral sequence, this group would contain torsion. As $\pi_0Tmf\!\left[\tfrac13\right] \cong \pi_0tmf\!\left[\tfrac13\right]$, it suffices to show that $\pi_0tmf \cong \Z$. It was known by Hopkins and Miller and is shown in \cite[Corollary 5.3]{Mattmf} that the Adams--Novikov spectral sequence for $tmf$ has as $E^2$-term the cohomology of the graded Weierstrass Hopf algebroid 
$$(A = \Z[a_1,a_2,a_3,a_4,a_6], \Gamma = A[r,s,t]).$$
Here, $|r| = 4, |s| = 2$ and $|t| =6$. It follows formally from the gradings in the cobar complex that $H^i(A,\Gamma) = 0$ in degrees smaller than $2i$ and that $H^0(A,\Gamma) \cong \Z$. The result follows.) 
\end{proof}

As $4\nu$ in $\pi_3Tmf\!\left[\tfrac13\right]$ is non-zero and of filtration $3$, we know that $\nu = \res(\nu) \in\pi_3 Tmf_0(3)$ is of filtration $\leq 3$ and non-zero. For degree reasons, it has to be detected by the image of $a_\sigma^3\ab_3$. As the homotopy fixed point spectral sequences for $tmf_1(3)^{hC_2}$ and $Tmf_1(3)^{hC_2}$ agree in this range, the same is true for $tmf_1(3)^{hC_2}$.

\begin{corollary}
The classes $\ab_{1}$ and $\ab_{3}$ are permanent cycles in the $RO(C_2)$-graded homotopy fixed point spectral sequence for $tmf_1(3)$. 
\end{corollary}
\begin{proof}
Since the homotopy fixed point spectral sequence and $a_{\sigma}$-Bockstein spectral sequences coincide, we learn that if an $a_{\sigma}$-multiple of a class is a permanent cycle, then the class is a permanent cycle. This in particular applies to $\eta=\ab_{1}a_{\sigma}$ and $\nu=\ab_{3}a_{\sigma}^{3}$.
\end{proof}

\begin{corollary}
The class $u_{2\sigma}$ is the only generator of the $E_{2}$ term for the $RO(C_{2})$-graded homotopy fixed point spectral sequences for $tmf_{1}(3)$ (as listed in Equation \ref{eq:E2term}) that is not a permanent cycle.
\end{corollary}

Furthermore, the transfer of any element in the underlying homotopy is a permanent cycle. In particular, we conclude immediately that the classes
\[
v_{0}(k):=2u_{2\sigma}^{k}
\]
for $k\in\mathbb Z$ are all permanent cycles which generate copies of $\mathbb Z$. These satisfy an obvious multiplicative relation
\[
v_{0}(k)v_{0}(j)=2 v_{0}(j+k).
\]

Next, we will determine the differentials. Note first that for degree reasons all $d_{2k}$ are $0$ for $k \geq 1$. While the other differentials could be deduced from \cite{M-R09}, we will derive them independently.

\begin{proposition}
We have the differential
\begin{align*}
d_3(u_{2\sigma}) &= a_{\sigma}^{3}\ab_{1}.
\end{align*}
\end{proposition}
\begin{proof}
 As $\ab_1, \ab_3$ and $a_\sigma$ are permanent cycles, $d_3(u_{2\sigma}) = 0$ would imply that $E_2 = E_5$. On the other hand, we know that $\eta$ is detected by $a_\sigma\ab_1$. As $\eta^4 = 0$, the class $(a_\sigma\ab_1)^4$ must be hit by a differential, which necessarily must be a $d_3$. Therefore, $d_3(u_{2\sigma}) \neq 0$. For degree reasons we get that $d_3(u_{2\sigma}) = a_{\sigma}^{3}\ab_{1}$.
\end{proof}

There is no room for a $d_5$-differential; indeed, a non-trivial $d_5$-differential would imply a differential of the form $d_5(u_{2\sigma}^2) = a_\sigma^5 y$ with $y$ in the $0$-line of degree $3+\sigma$, which is impossible. Thus, $E_7 = E_4$.

\begin{proposition}
 We have the differential
 \[d_7(u_{2\sigma}^{2}) = a_{\sigma}^{7}\ab_{3}.\]
\end{proposition}
\begin{proof}
If $d_n(u_{2\sigma}^2) = a_\sigma^nx$, then $x$ is in degree $(7-n) + (n-4)\rho$. As $x$ can be written as $u_{2\sigma}^{2m}$ times a polynomial in $\ab_1$ and $\ab_3$, we see that $7-n$ must be divisible by $8$. As $\ab_1, \ab_3$, $a_\sigma$ and $2u_{2\sigma}$ are permanent cycles, $d_7(u_{2\sigma}^2) = 0$ would thus imply that $E_7 = E_{15}$. On the other hand, we know that $\nu$ is detected by $a_\sigma^3\ab_3$. As $\nu^4 = 0$, the class $(a_\sigma^3\ab_3)^4$ must be hit by a differential, which necessarily must be a $d_n$ with $n\leq 12$. Therefore, $d_7(u_{2\sigma}^2) \neq 0$. For degree reasons we get that $d_3(u_{2\sigma}^2) = a_{\sigma}^{7}\ab_{3}$ (as $a_\sigma^7\ab_1^3 = 0$ in $E_7$).
\end{proof}

The torsion produced by the first differential yields new $d_{7}$-cycles:
\[
\ab_{1}(k):=\ab_{1}u_{2\sigma}^{2k},
\]
for $k\in\mathbb Z$.
These also participate in the expected multiplicative relations:
\begin{align*}
\ab_1(k)\ab_{1}(j)&=\ab_{1} \cdot\ab_{1}(j+k),\\
\ab_1(j) v_{0}(k) &= \ab_{1}\cdot v_{0}(k+2j).
\end{align*}

\begin{remark}
The classes $v_{0}(k)$ and $\ab_{1}(j)$ form families exactly like the families $v_{0}(k)$ and $v_{1}(j)$ described by Hu-Kriz is the computation of the homotopy of $BP\mathbb{R}$.
\end{remark}

There is no room for further differentials in $E_8$, which is the subalgebra of
\[
\Z\big[\tfrac13\big][a_\sigma, u_{2\sigma}, \ab_1,\ab_3]/(2a_\sigma, \ab_1a_{\sigma}^3, \ab_3a_{\sigma}^7)
\]
 generated by $a_\sigma$, $\ab_1$, $\ab_3$, $v_0(1)$, $v_0(2)$, $v_0(3)$, $\ab_1(1)$ and $u_{2\sigma}^{\pm 4}$. Indeed, a non-trivial $d_k$-differential for $k\geq 8$ would imply a non-trivial differential of the form $d_k(\ab_1(1)) = a_\sigma^k x$ or $d_k(u_{2\sigma}^4) = a_\sigma^ky$ for some $x$ or $y$ in the $0$-line of degree $4+(k-3)\sigma$ or $7+(k-8)\sigma$, respectively; but the only of our generators of the $0$-line not killed by $a_\sigma^k$ is $u_{2\sigma}^{\pm 4}$ whose powers cannot be in degree $4+(k-3)\sigma$ or $7+(k-8)\sigma$. Therefore $E_8 = E_\infty$.

\begin{thm}\label{thm:computation}
We have
\[\pi^{C_2}_\bigstar F\big(EC_{2+},tmf_1(3)\big) \cong \Z\big[\tfrac13\big][a_\sigma, u_{2\sigma}^{\pm 4}, \ab_1,\ab_3, v_{0}(k), \ab_1(1)]/R,\]
where the ideal $R$ of relations is generated by
\begin{align*}
 a_{\sigma}v_{0}(k) &= 0 \\
 a_{\sigma}^3\big(\ab_1, \ab_1(1)\big) &= 0 \\
 a_\sigma^7\ab_3 &= 0 \\
 v_{0}(k+4)&=v_{0}(k)u_{2\sigma}^{4}\\
 v_{0}(k)v_{0}(j)&=2 v_{0}(j+k)\\
 \ab_1(1) v_{0}(k) &= \ab_{1} v_{0}(k+2) \\
 \ab_1(1)^2 &= \ab_1u_{2\sigma}^4.
\end{align*}
\end{thm}
\begin{proof}
 The presentation given was already shown to be a presentation of the $E_\infty$-term. We just have to check all the relations also to hold in $\pi^{C_2}_\bigstar F\big(EC_{2+},tmf_1(3)\big)$. Observe first that no two torsion classes in different filtrations can converge to the same bidegree. This implies the first three relations must hold. In the next three relations, both sides are in the image of the transfer and thus these relations can be checked on underlying homotopy groups. The last relation holds again since there is no element of filtration $\geq 1$ in this bidegree.
\end{proof}

\begin{remark}
We have $\pi^{C_2}_{a+b\sigma} F\big(EC_{2+},tmf_1(3)\big) \cong \pi^{C_2}_{a+b\sigma} tmf_1(3)$ for all $a,b\geq 0$ and $\pi_{a+b\sigma}^{C_2}tmf_1(3) = 0$ if $a<0$ and $a+b<0$; this follows from the cofiber sequence
\[S^{a+(b-1)\sigma} \to S^{a+b\sigma}\to S^{a+b}\sm (C_2)_+.\] Is it possible, but more complicated, to describe also the other groups in $\pi_\bigstar^{C_2}tmf_1(3)$. Note that $\pi^{C_2}_i F\big(EC_{2+},tmf_1(3)\big) \cong \pi_i tmf_1(3)^{hC_2}$ can certainly be non-trivial for $i<0$ (e.g.\ we have $a_\sigma^8u_{2\sigma}^{-4} \in \pi_{-8}tmf_1(3)^{hC_2}$) and thus $tmf_1(3) \not\simeq F\big(EC_{2+},tmf_1(3)\big)$.
\end{remark}

\begin{cor}\label{cor:StronglyEven}
The spectrum $tmf_1(3)$ is strongly even as a $C_2$-spectrum. In particular, it is Real orientable and thus $tmf_1(3)_{(2)}$ is a form of $BP\R\langle 2\rangle$. Furthermore, $tmf_1(3)_{(2)}[\ab_3^{-1}]$ is a form of $E\R(2)$. 
\end{cor}
\begin{proof}
It follows from Theorem \ref{thm:computation} and the remark thereafter that $tmf_1(3)$ is even as a $C_2$-spectrum and also that the Mackey functor $\m{\pi}_{k\rho}tmf_1(3)$ is constant for all $k\in\Z$. We present the argument for evenness and leave the other part to the reader. Let $y = a_\sigma^lx$ be a nonzero class in degree $k\rho -1$ with $x$ of filtration $0$ and degree $(k-1) + (k+l)\sigma$. Clearly $l\geq 1$. In $E_2$, we can write $x$ as $\ab_1^i\ab_3^ju_{2\sigma}^{2m}$. We see that $(k-1) - (k+l) = -(l+1)$ is divisible by $8$; in particular, $l\geq 7$. This implies $i,j=0$ and leads to a contradiction.

In the following, we localize everywhere implicitly at $2$. The map $BP_* \to tmf_1(3)_*$ induced by the $2$-typification of the formal group law associated to the Weierstrass equation $y^2+a_1xy+a_3y = x^3$ sends the Hazewinkel generators $v_1$ and $v_2$ exactly to $a_1$ and $a_3$. This implies together with $tmf_1(3)$ being stronlgy even that $tmf_1(3)$ is a form of $BP\R\langle 2\rangle$ and that $tmf_1(3)[\ab_3^{-1}]$ is a form of $E\R(2)$.
\end{proof}

\begin{cor}
 There exists forms of $BP\R\langle 2\rangle$ and $E\R(2)$ that are strictly commutative $C_2$-ring spectrum.
\end{cor}
\begin{proof}
 By Theorem \ref{thm:connective}, the spectrum $tmf_1(3)$ has the structure of a strictly commutative $C_2$-ring spectrum. By the last corollary, it is a form of $BP\R\langle 2\rangle$. 
 
 As shown in the next section, the spectrum $tmf_1(3)[\ab_3^{-1}]$ is equivalent to $Tmf_1(3)[\ab_3^{-1}]$ and thus cofree. Thus, we see by Theorem \ref{thm:cofree} that it has the structure of a strictly commutative $C_2$-spectrum.
\end{proof}
\begin{remark}
 We do not know whether the forms of $BP\R\langle 2\rangle$ and $E\R(2)$ exhibited here are equivalent as $C_2$-spectra to other known forms, as for example defined via the Hazewinkel generators. Note though that two forms of $E\R(n)$ are equivalent if and only if their underlying spectra are equivalent by Proposition \ref{prop:JWUniqueness}. 
Note further that in contrast to our result, the existence of any kind of (homotopy unital) multiplication seems to be unknown for general forms of $BP\R\langle n\rangle$, even for $n=2$.    
\end{remark}

\subsection{The relationship between \texorpdfstring{$tmf_1(3)$, $Tmf_1(3)$ and $TMF_1(3)$}{tmf13, Tmf13, and TMF13}}\label{sec:relationship}
The following is proven in \cite[Theorem 7.12]{MM15}.
\begin{prop}\label{prop:Galois}
The map $Tmf_0(3) \to Tmf_1(3)$ is a faithful $C_2$-Galois extension in the sense of Rognes.
\end{prop}

\begin{lemma}\label{lem:invert}
Let $\overline{f}$ be a nonconstant homogeneous polynomial in $\ab_1$ and $\ab_3$. Then
\[
tmf_1(3)[\overline{f}^{-1}] \to Tmf_1(3)[\overline{f}^{-1}]
\]
is an equivalence.
\end{lemma}
\begin{proof}
For some $k>0$, we have $a_\sigma^7\overline{f}^k = 0$ in $\pi^{C_2}_\bigstar F((EC_2)_+, tmf_1(3))$ and $|a_\sigma^7\overline{f}^k| = r+s\sigma$ with $r,s\geq 0$. Hence also $a_\sigma^7\overline{f}^k = 0$ in $\pi_\bigstar^{C_2} tmf_1(3)$ and thus $\Phi^{C_2}(tmf_1(3)[\overline{f}^{-1}]) = 0$ by Lemma \ref{lem:Geometric}. By \cite[Corollary 10.6]{HHR09}, $tmf_1(3)[\overline{f}^{-1}]$ is then cofree. Thus, we have only to show that $tmf_1(3)[\overline{f}^{-1}] \to Tmf_1(3)[\overline{f}^{-1}]$ is an equivalence of underlying spectra. As every element of negative degree in $\pi^e_*Tmf_1(3)$ is killed by $a_1$ and $a_3$, the result follows.
\end{proof}

\begin{lemma}\label{lem:invert2}
Let $\overline{f}$ be a nonconstant homogeneous polynomial in $\ab_1$ and $\ab_3$. Denote by $D(f)$ the non-vanishing locus of the underlying element $f\in H^0(\MMb_1(3); \omega^*)$. Then there is an equivalence
\[
Tmf_1(3)[\overline{f}^{-1}] \to \OO^{top}(D(f))
\]
of $C_2$-spectra.
\end{lemma}
\begin{proof}
Note that the pullback of $D(f)$ along
\[
\Spec \Z[\tfrac13][a_1,a_3] - \{0\}\quad \longrightarrow \quad \MMb_1(3) \simeq \PP_{\Z[\tfrac13]}(1,3)
\]
is $\Spec \Z[\tfrac13][a_1,a_3][f^{-1}].$ By the same argument as in Lemma \ref{lem:Landweber}, the global sections functor
\[
\Gamma: \QCoh(D(f)) = \QCoh\Big(\Spec \big(\Z[\frac13][a_1,a_3][f^{-1}]\big)/\G_m\Big) \to \mathrm{Abelian Groups}
\]
is exact. Therefore, the descent spectral sequence for $\OO^{top}(D(f))$ collapses and we have $\pi_*\OO^{top}(D(f)) \cong \Z[\frac13][a_1,a_3][f^{-1}]$.

Note furthermore that $D(f)$ is $C_2$-invariant as $f^2$ is an invariant section. This induces a $C_2$-map of ring spectra $Tmf_1(3) = \OO^{top}(\MMb_1(3)) \to \OO^{top}(D(f))$. We want to show that the image of $\overline{f}$ is invertible in $\pi_\bigstar^{C_2} \OO^{top}(D(f))$. It is detected in the homotopy fixed point spectral sequence $\mathrm{HFPSS}(f)$ for $\OO^{top}(D(f))^{hC_2}$ by $fu_{2\sigma}^k$ for some $k$. As $f$ and $u_{2\sigma}$ are invertible, there exists an element 
\[
\overline{g} \in \pi^{C_2}_{-|\overline{f}|}\OO^{top}(D(f))
\]
detected by $f^{-1}u_{2\sigma}^{-k}$. Clearly, we have that the underlying class $\mathrm{res}(\overline{f}\overline{g}) \in \pi_0\OO^{top}(D(f))$ equals $1$. As $\mathrm{HFPSS}(f)$ receives a multiplicative map from the homotopy fixed point spectral sequence $\mathrm{HFPSS}$ for $tmf_1(3)^{hC_2}$, the identification of $\pi_*\OO^{top}(D(f))$ above implies that $\mathrm{HFPSS}(f) \cong \mathrm{HFPSS}[\overline{f}^{-1}]$. In particular, we can deduce that $\OO^{top}(D(f))$ is strongly even as a (cofree) $C_2$-spectrum. This implies that $\overline{f}\overline{g} = 1 \in \pi_0^{C_2}\OO^{top}(D(f))$ so that $\overline{f}$ is invertible. Thus, we get an induced map
\[Tmf_1(3)[\overline{f}^{-1}] \to \OO^{top}(D(f))\]
of $C_2$-spectra.

By \cite[Theorem 7.2]{MM15} and the proof of \cite[Theorem 7.12]{MM15}, the global sections functor
\[\Gamma\co \QCoh(\MMb_1(3),\OO^{top}) \to Tmf_1(3)\modules\]
is an equivalence.\footnote{We only really need that $\Gamma$ commutes with homotopy colimits. As observed in the proof of \cite[Proposition 3.8]{MM15}, this is automatic when the stack has finite cohomological dimension as $\MMb_1(3)$ does. This circumvents the use of most of the heavy machinery in \cite{MM15}.} Thus, we can apply \cite[Lemma 3.20]{MM15} to see that
\[Tmf_1(3)[\overline{f}^{-1}] \to \OO^{top}(D(f)) \]
is an equivalence of underlying spectra. As both spectra are cofree, the result follows.
\end{proof}

This applies in particular to $\overline{f} = \overline{\Delta}$. Thus,
\[tmf_1(3)[\overline{\Delta}^{-1}] \simeq Tmf_1(3)[\overline{\Delta}^{-1}] \simeq TMF_1(3)\]
as $C_2$-spectra (with $\overline{\Delta} = \overline{a_3}^3(\overline{a_1}^3-27\overline{a_3}))$. In particular, $TMF_1(3)$ is strongly even. Thus, Theorem \ref{thm:RealLandweber} implies:
\begin{prop}
The $C_2$-spectrum $TMF_1(3)$ is Real Landweber exact in the sense that there is a natural isomorphism
\[M\R_\bigstar(X) \tensor_{MU_{2*}} TMF_1(3)_{2*} \to TMF_1(3)_\bigstar(X)\]
for all $C_2$-spectra $X$.
\end{prop}
Note that the equivalence $tmf_1(3)[\overline{\Delta}^{-1}] \simeq_{C_2} TMF_1(3)$ also directly implies together with the computations from the previous sections that $\pi_*TMF_0(3)$ has torsion and thus $TMF_0(3)$ cannot be Landweber exact.

The following fiber square will be useful later.
\begin{prop}\label{prop:FiberSquare}
We have a fiber square
\[
\xymatrix{ Tmf_1(3) \ar[r]\ar[d] & tmf_1(3)[\ab_1^{-1}] \ar[d] \\
tmf_1(3)[\ab_3^{-1}] \ar[r] & tmf_1(3)[(\ab_1\ab_3)^{-1}] }
\]
\end{prop}
\begin{proof}
The square
\[
\xymatrix{ \MMb_1(3) & D(a_1) \ar[l] \\
D(a_3) \ar[u] & D(a_1a_3) \ar[u]\ar[l] }
\]
induces a fiber square
\begin{align}\label{FiberSquare}
\xymatrix{\OO^{top}(\MMb_1(3)) \ar[r]\ar[d] & \OO^{top}(D(a_1)) \ar[d] \\
\OO^{top}(D(a_3)) \ar[r] & \OO^{top}(D(a_1a_3)) }
\end{align}
as
\[
\MMb_1(3) \simeq \PP_{\Z\big[\tfrac13\big]}(1,3) = D(a_1)\cup D(a_3)
\]
and $\OO^{top}$ is a sheaf (see \cite[Appendix A]{MM15} for why the sheaf condition implies this).

By the last two lemmas, this is equivalent to
\[
\xymatrix{ Tmf_1(3) \ar[r]\ar[d] & tmf_1(3)[\ab_1^{-1}] \ar[d] \\
tmf_1(3)[\ab_3^{-1}] \ar[r] & tmf_1(3)[(\ab_1\ab_3)^{-1}] }
\]
as a square of $C_2$-spectra.
\end{proof}

\section{Slices and Anderson Duals}\label{sec:SliceAnderson}
In this section, we will compute the slices of $TMF_1(3)$ and $Tmf_1(3)$ and apply this to compute the Anderson dual of $Tmf_1(3)$. 
\subsection{Slices}
We can apply the computations of the regular representation homotopy groups of $tmf_{1}(3)$ and its localizations to determine their slices.

Since all of the odd slices vanish and the even slices are regular representation suspensions of $H\mZ\left[\tfrac13\right]$ by \ref{sec:Slices} and \ref{cor:StronglyEven}, the homotopy groups ``near multiples of regular representations'' are easy to compute since the slice spectral sequence is especially simple here.

To ensure transparency with later notation and gradings, we introduce some notation. Let $R=\Z[\tfrac13][a_1,a_3][f^{-1}]$ with $f$ homogeneous. If $S\subset R_{2n}$ is any subset of homogeneous rational functions of degree $2n$, then let $\overline{S}$ denote the same rational functions, but with every instance of $a_1$ and $a_3$ replaced with $\ab_1$ and $\ab_3$ respectively. This is a notational device to ensure that the reader keep track of the $RO(C_2)$-grading of barred elements, compared to the underling, $\mathbb Z$-grading of unbarred ones. Lemma \ref{lem:HomotopyGroups} now gives us a description of the homotopy groups of the localizations of $tmf_1(3)$:

\begin{corollary}\label{cor:HomotopyGroups}
Let $M$ be one of $tmf_1(3)$, $tmf_1(3)[\ab_1^{-1}]$, $tmf_1(3)[\ab_3^{-1}]$, or $tmf_1(3)[(\ab_1\ab_3)^{-1}]$.

For all $k$, we have
\begin{align*}
 \m{\pi}_{k\rho+1} M&= \m{G}\otimes \overline{\pi_{2k+2}^u M}, \\
 \m{\pi}_{k\rho} M&=\mZ\left[\tfrac13\right]\otimes\overline{\pi_{2k}^u M}, \\
 \m{\pi}_{k\rho-1} M&=0, \text{ and } \\
 \m{\pi}_{k\rho-2} M&=\mZ\left[\tfrac13\right]_{-}\otimes\overline{\pi_{2k-2}^uM}.
\end{align*}
\end{corollary}

Similarly, naturality of the slice spectral sequence implies that we understand the effect of the localization maps on homotopy groups in dimensions $k\rho-2,\dots,k\rho+1$.

\begin{corollary}
For $k\in\Z$ and for $j=-2, -1, 0, 1$, the localization maps
\[
\m{\pi}_{k\rho+j} tmf_1(3)[\ab_i^{-1}]\to\m{\pi}_{k\rho+j} tmf_1(3)[(\ab_1\ab_3)^{-1}]
\]
are induced by the obvious inclusions of graded pieces of these graded rings.
\end{corollary}
\begin{remark}
We could also have read off these results from the homotopy fixed point spectral sequence, but the slice spectral sequence approach is both more conceptual and is easier for Mackey functor computations.
\end{remark}

We want now to compute the slices of $Tmf_1(3)$. To that purpose, we denote by $M[\ab_1,\ab_3]$ the monic monomials in $\Z\big[\tfrac13\big][\ab_1,\ab_3]$.
\begin{prop}\label{prop:SliceGraded}
The associated graded $C_2$-spectrum for the slice filtration of $Tmf_1(3)$ is
\[\bigvee_{P\in M[\ab_1, \ab_3]} S^{|P|}\sm H\underline{\Z}\left[\tfrac13\right]\quad \vee \bigvee_{P\in M[\ab_1, \ab_3]} S^{-|P|-4\rho-1} \sm H\underline{\Z}\left[\tfrac13\right].\]
\end{prop}
\begin{proof}
%
We use Propositions~\ref{prop:OddSlices} and \ref{prop:EvenSlices} to read the slices out of the $RO(C_2)$-graded homotopy groups. The long exact sequence in homotopy associated to the fiber square \ref{prop:FiberSquare} and Corollary~\ref{cor:HomotopyGroups} identify the needed homotopy groups. For $k<0$, let $R_{k}$ denote the degree $2k$ piece of
\[
\Z\big[\tfrac13\big][a_{1}^{\pm 1},a_{3}^{\pm 1}]/\left(\Z\big[\tfrac13\big][a_{1}^{\pm 1},a_{3}]+\Z\big[\tfrac13\big][a_{1},a_{3}^{\pm 1}]\right).
\]
We then have isomorphisms
\[
\m{\pi}_{k\rho} Tmf_{1}(3)=\m{G}\tensor R_{k+1}
\]
and
\[
\m{\pi}_{k\rho-1} Tmf_{1}(3)=\mZ\tensor R_{k}.
\]
The functor $P^{0}$ applied to the Mackey functor $\m{G}$ yields zero, so we conclude by Proposition \ref{prop:EvenSlices} that there are no negative even slices, and and by Proposition \ref{prop:OddSlices} that all of the negative odd slices are of the desired form.
\end{proof}

This allows us to compute the $E_2$-term of the slice spectral sequence
\[
E_2^{s,t} = \pi_{t-s}^{C_2}P^t_tTmf_1(3) \Rightarrow \pi_{t-s}^{C_2}Tmf_1(3),
\]
where $P^t_t$ denotes the $t$-slice of $Tmf_1(3)$. For $t=2k\geq 0$, we get:
\begin{align*}
\pi_{2k-s}^{C_2}P^{2k}_{2k}Tmf_1(3)&= \bigoplus_{P\in M[\ab_1, \ab_3]_{k\rho}}\pi^{C_2}_{2k-s} S^{k\rho} \sm H\underline{\Z}\left[\tfrac13\right]\\
&= \bigoplus_{P\in M[\ab_1, \ab_3]_{k\rho}} H_{2k-s}^{C_2}\left(S^{k\rho}, \m{\Z}\left[\tfrac13\right]\right) \\
&= \bigoplus_{P\in M[\ab_1, \ab_3]_{k\rho}} H_{k-s}^{C_2}\left(S^{k\sigma},\m{\Z}\left[\tfrac13\right]\right)
\end{align*}
By \cite[Example 3.16]{HHR09}, we have:
\[
H_{k-s}^{C_2}\left(S^{k\sigma}, \m{\Z}\left[\tfrac13\right]\right) = \begin{cases} \Z\big[\tfrac13\big] & \text{ if } 2k-s \text{ divisible by }4\text{ and }s=0 \\
\Z/2 & \text{ if } 0<s\leq 2k-s \text{ and }(2k-s)-s \text{ divisible by } 4 \\
0 & \text{ else}\end{cases}
\]

Similarly, one can reduce the computation for $t<0$ to Bredon \emph{co}homology and use that
\[H^k_{C_2}\left(S^{d\sigma}, \m{\Z}\left[\tfrac13\right]\right) = \begin{cases} \Z\big[\tfrac13\big] & \text{ if } d \text{ even and }k=d \\
\Z/2 & \text{ if } k \text{ odd and }1<k\leq d \\
0 & \text{ else}\end{cases}\]

We depict the slice spectral sequence in Figure~\ref{fig:SlSpSq}. Here, the unboxed number $n$ denotes $n$ copies of $\Z/2$, a box denotes a copy of $\Z\big[\tfrac13\big]$ and a boxed $n$ denotes $n$ copies of $\Z\big[\tfrac13\big]$. The vertical coordinate is $s$ and the horizontal one is $t-s$. In positive degrees, the differentials follow from those for $MU\mathbb R$ via the Real orientation map, and these were determined in \cite{HHR09} and in \cite{H-K01}. For the differentials in negative degrees, we can use that this is a spectral sequence of algebras, so in particular, we have an action of the slice spectral sequence for $tmf_1(3)$ on that of $Tmf_1(3)$. This reduces the problem to understanding the differentials on the line $L$ of slope one in Figure~\ref{fig:SlSpSq} passing through the ``1'' in $(-8,-1)$. This class is infinitely divisible by $\eta=\ab_1a_\sigma$ and $\nu=\ab_3a_\sigma^{3}$. The classes $\eta^3$ and $\nu^3$ are hit by a $d_3$ and a $d_7$ respectively in the slice spectral sequence 
for $tmf_1(3)$. As a class $x$ on $L$ is not hit by any differential for degree reasons, $\eta^{-3}x$ has thus to support a $d_3$-differential and $\nu^{-3}x$ a $d_7$-differential (if it does not support a $d_3$-differential). This forces the negative differentials.

\begin{figure}\label{fig:SlSpSq}
\centering
\includegraphics[scale=0.6, angle=90]{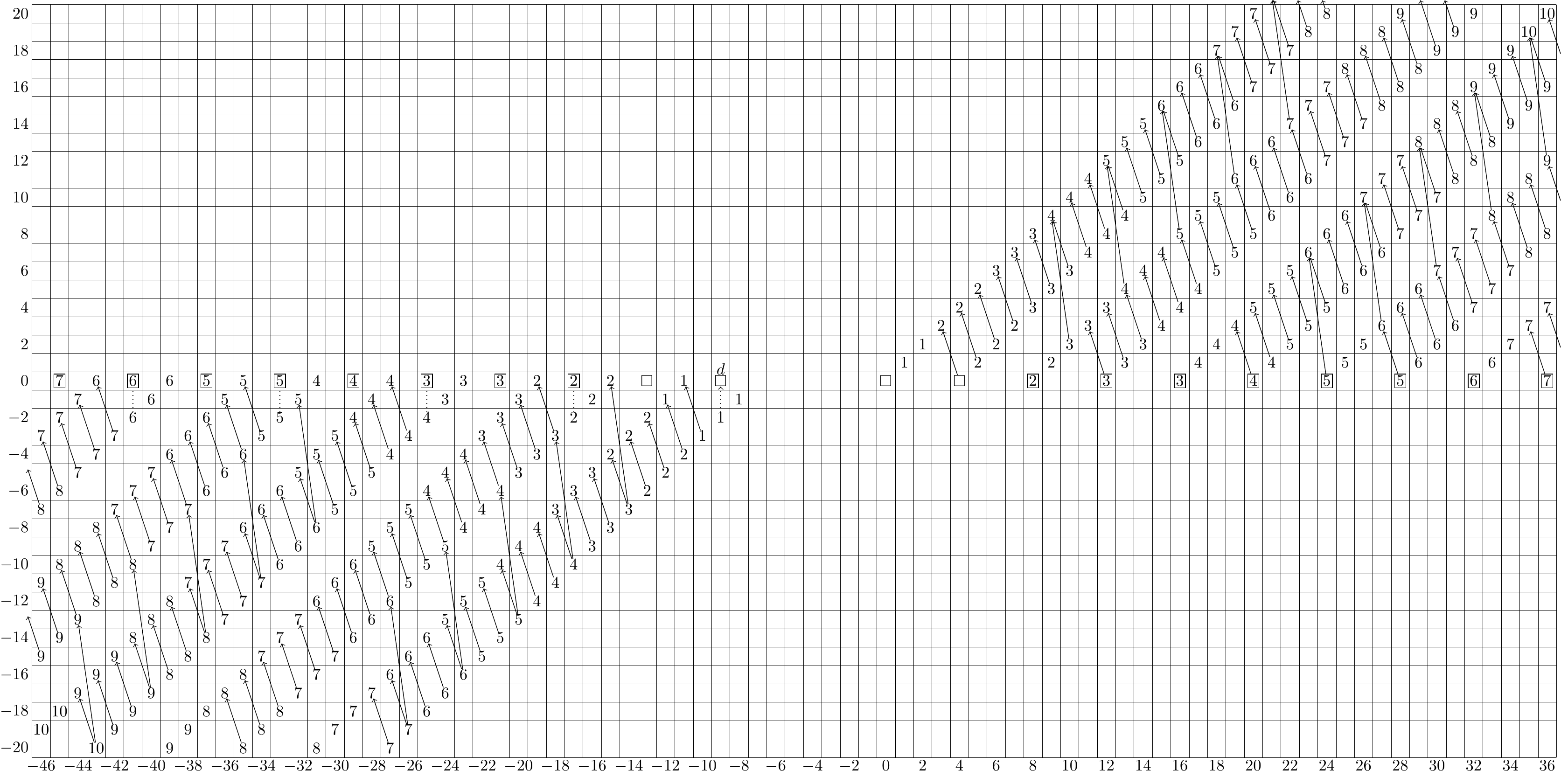}
\caption{The slice spectral sequence for $Tmf_1(3)$.}
\end{figure}

\subsection{Anderson Duality}
Let $G$ be a finite group. For an injective abelian group $J$, the functor
\[\text{(genuine) } G-\mathrm{Spectra} \to \text{graded abelian groups},\quad X \mapsto \Hom_\Z(\pi^G_{-*}X, J)\]
is representable by a $G$-spectrum $I_J$, as follows from Brown representability. If $A$ is an abelian group and $A \to J^0 \to J^1$ an injective resolution, we define the $G$-spectrum $I_A$ to be the fiber of $I_{J^0}\to I_{J^1}$. Given a $G$-spectrum $X$, we define its \emph{$A$-Anderson dual} $I_AX$ by $F(X, I_A)$. It satisfies for all $k\in\Z$ the following functorial short exact sequence:
\[0 \to \Ext^1_\Z(\pi^G_{-k-1}X, A) \to \pi^G_kI_AX \to \Hom_\Z(\pi_{-k}^GX, A) \to 0\]

For $G = \{e\}$ we get non-equivariant Anderson duality as explored in \cite{And70} and \cite{Sto12}. If $G$ is (possibly) non-trivial, denote by $\AA_G$ the \emph{stable Burnside category}, by which we mean the full subcategory of $\Ho(\Sp_G)$ on the cosets $\Sigma^\infty(G/H)_+$. Given again a $G$-spectrum $X$, we see by precomposing with the functor
\[\AA_G \to \Sp_G,\quad \Sigma^\infty(G/H)_+ \mapsto \Sigma^\infty(G/H)_+ \sm X\]
that the short exact sequence above refines to a short exact sequence of Mackey functors
\[0 \to \Ext^1_\Z(\mpi_{-k-1}X, A) \to \mpi_kI_AX \to \Hom_\Z(\mpi_{-k}X, A) \to 0.\]
By smashing $X$ with representation spheres, we see that it even refines to an $RO(G)$-graded sequence. Equivariant Anderson duality in the case $G=C_2$ has been explored in some detail in \cite{Ric14}.

One reason to be interested in Anderson (self) duality is the universal coefficient sequence, relating homology and cohomology. Let $E$ be a $G$-spectrum, $X$ be another $G$-spectrum and $A$ be an abelian group. Then $I_A(X\sm E) \simeq_G F(X, I_AE)$ implies the short exact sequence
\[0 \to \Ext_\Z^1(E_{V-1}X, A) \to (I_AE)^VX \to \Hom_\Z(E_VX, A) \to 0\]
for a real $G$-representation $V$. In particular, Anderson self-duality implies useful universal coefficient theorems; for example, $I_{\Z}KO \simeq \Sigma^4KO$ implies one of the main theorems of \cite{And70}.

Our goal in this section is to compute the $\Z\big[\tfrac13\big]$-Anderson dual of $Tmf_1(3)$ as a $C_2$-spectrum and then deduce a computation of the $\Z\big[\tfrac13\big]$-Anderson dual of $Tmf_0(3)$.\\

Observe that $H\underline{\Z}^* \simeq S^{4-2\rho} \sm H\underline{\Z}$ as $H^0(S^{4-2\rho}; \m{\Z}) \cong \m{\Z}^*$, where $\m{\Z}^*$ is as in Definition \ref{def:Mackey}. Thus, Proposition \ref{prop:SliceGraded} implies that the associated graded $C_2$-spectrum for the slice filtration of $Tmf_1(3)$ is
\[\bigvee_{P\in M[\ab_1, \ab_3]} S^{|P|}\sm H\underline{\Z}\left[\tfrac13\right]\quad \vee \bigvee_{P\in M[\ab_1, \ab_3]} S^{-|P|-2\rho-5} \sm H\underline{\Z}\left[\tfrac13\right]^*.\]
This suggests the following theorem:

\begin{thm}\label{thm:Anderson}
There is a $C_2$-equivariant equivalence $I_{\Z\big[\tfrac13\big]}Tmf_1(3) \simeq \Sigma^{5+2\rho} Tmf_1(3)$.
\end{thm}
Note that this theorem implies the universal coefficient sequence claimed in the introduction. To prove the theorem, we will start with two lemmas.

\begin{lemma}\label{lem:AndersonDual}
We have non-equivariantly $I_{\Z\big[\tfrac13\big]}Tmf_1(3) \simeq \Sigma^9 Tmf_1(3)$.
\end{lemma}
\begin{proof}
By Proposition \ref{prop:StackIdentification}, the moduli stack $\MMb_1(3)$ is equivalent to the weighted projective stack $\PP(1,3) = \PP_{\Z\big[\tfrac13\big]}(1,3)$ and the sheaf $\omega$ on $\MMb_1(3)$ corresponds to $\OO(1)$ on $\PP(1,3)$. This weighted projective stack has Serre duality in the sense that there is a class 
\[
D = \frac1{a_1a_3} \in H^1\big(\PP(1,3); \OO(-4)\big)
\]
such that
\[
H^s\big(\PP(1,3); \FF\big) \tensor H^{1-s}\big(\PP(1,3); \FF^* \tensor \OO(-4)\big) \to H^1\big(\PP(1,3); \OO(-4)\big) \cong\Z\big[\tfrac13\big]\cdot D
\]
is a perfect pairing for $s=0,1$ for an arbitrary coherent sheaf $\FF$.

Let us write for brevity $R = Tmf_1(3)$. As $\PP(1,3)$ has cohomological dimension $1$, the element $D$ is a permanent cycle in the descent spectral sequence for $R$ and is represented by a unique element in $\pi_{-9}R \cong \Z\big[\tfrac13\big]$, which we will also denote by $D$. Denote by $\delta$ the element in $\pi_9 I_{\Z\big[\tfrac13\big]} R$ with $\phi(\delta)(D) = 1$, where 
\[
\phi\colon \pi_9 I_{\Z\!\big[\tfrac13\big]} R \xrightarrow{\cong} \Hom(\pi_{-9}R,\Z\big[\tfrac13\big]).
\]
The element $\delta$ induces a $R$-linear map $\widehat{\delta}\co \Sigma^9R\to I_{\Z\big[\tfrac13\big]} R$.

We obtain a commutative diagram
\[\xymatrix{\pi_{k-9}R \tensor \pi_{-k}R \ar[rr]^-{\widehat{\delta}_*\tensor \id}\ar[d] && \pi_kI_{\Z\big[\tfrac13\big]}R\tensor \pi_{-k}R \ar[rr]_-{\cong}^-{\phi\tensor\id} && \Hom(\pi_{-k}R, \Z\big[\tfrac13\big]) \tensor \pi_{-k}R\ar[d] \\
\pi_{-9}R\ar[rrrr]^{\phi(\delta)}_{\cong} &&&& \Z\big[\tfrac13\big] }
\]

The left vertical map is a perfect pairing because of Serre duality (as described above), as is the right vertical map by definition. Thus, the map $\widehat{\delta}_*\colon \pi_{k-9}R \to \pi_kI_{\Z\big[\tfrac13\big]}R$ is an isomorphism for all $k$. This shows that $\widehat{\delta}$ is an equivalence.
\end{proof}

The following key lemma uses our information about the slices of $Tmf_1(3)$:
\begin{lemma}
The transfer
\[\pi_{-9} Tmf_1(3) = \pi^{e}_{-5-2\rho}Tmf_1(3) \to \pi_{-5-2\rho}^{C_2}Tmf_1(3)\]
 is an isomorphism.
\end{lemma}
\begin{proof}
The slice spectral sequence for $\Sigma^{2\rho}Tmf$ (as shown in Figure~\ref{fig:tworhoslicess}, where dots stand for the Mackey functor $G$ and a box with a cross stands for $\underline{\Z}^*$) gives an isomorphism of Mackey functors
\[
\m{\pi}_{-5-2\rho}Tmf_1(3) \cong \m{\pi}_{-5-2\rho} S^{-4\rho-1}\sm H\underline{\Z}\left[\tfrac13\right] \cong  H^2\big(S^{2\sigma}; \underline{\Z}\left[\tfrac13\right]\big)\cong \mZ\left[\tfrac{1}{3}\right]^{\ast}.  \qedhere
\]
\end{proof}

\begin{figure}
\includegraphics{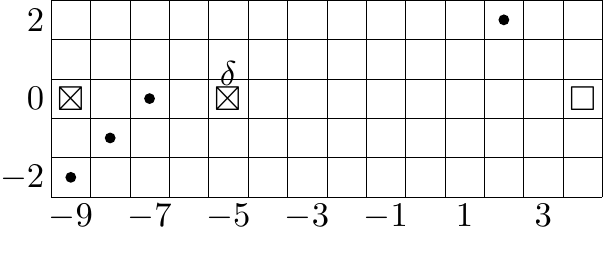}
\caption{The $E^2$-term of the Slice Spectral Sequence for $\pi_{k-2\rho} Tmf_{1}(3)$}
\label{fig:tworhoslicess}
\end{figure}

\begin{proof}[Proof of Theorem]
Consider the commutative diagram
\[\xymatrix{\pi^{C_2}_{5+2\rho}I_{\Z\big[\tfrac13\big]}Tmf_1(3) \ar[d]^{\mathrm{res}} \ar@{->>}[r] & \Hom(\pi_{-5-2\rho}^{C_2}Tmf_1(3), \Z\big[\tfrac13\big]) \ar[d]^{\mathrm{res} = \mathrm{tr}^*}_{\cong} \\
\pi_9 I_{\Z\big[\tfrac13\big]}Tmf_1(3) \ar[r]^-{\cong} & \Hom(\pi_{-9}Tmf_1(3), \Z\big[\tfrac13\big])}
\]
By the last lemma, $\mathrm{tr}^*$ is an isomorphism. This implies that we can refine the element $\delta \in \pi_9 I_{\Z\big[\tfrac13\big]}Tmf_1(3)$ corresponding to the equivalence $\Sigma^9Tmf_1(3) \to I_{\Z\big[\tfrac13\big]}Tmf_1(3)$ from Lemma \ref{lem:AndersonDual} to an element $\tilde{\delta} \in \pi^{C_2}_{5+2\rho}I_{\Z\big[\tfrac13\big]}Tmf_1(3)$. This induces a $C_2$-equivariant $Tmf_1(3)$-linear map
\[\Sigma^{5+2\rho} Tmf_1(3) \to I_{\Z\big[\tfrac13\big]}Tmf_1(3)\]
that is an equivalence of underlying spectra. By \ref{prop:Galois} and \cite[Prop 6.3.3]{Rog08}, we know that $Tmf_1(3)^{tC_2} \simeq \Phi^{C_2}Tmf_1(3)$ vanishes and thus $Tmf_1(3)$ and $I_{\Z\big[\tfrac13\big]}Tmf_1(3)$ are cofree $C_2$-spectra by \cite[Cor 10.6]{HHR09}. Thus, the theorem follows.
\end{proof}

This allows us also to compute the Anderson dual of $Tmf_0(3)$. As in \cite{Sto12}, we will use the following lemma:
\begin{lemma}
Let $A$ be an abelian group and $X$ be a spectrum with an action by a finite group $G$. Assume that the norm map $X_{hG} \to X^{hG}$ is an equivalence. Then there is an equivalence $(I_AX)^{hG} \simeq I_A(X^{hG})$.
\end{lemma}
\begin{proof}
We have the following chain of equivalences:
\[(I_AX)^{hG} \simeq F(X, I_A)^{hG} \simeq F(X_{hG}, I_A) \simeq F(X^{hG}, I_A) \simeq I_A(X^{hG})\qedhere\]
\end{proof}

As note in the proof of Theorem \ref{thm:Anderson}, $Tmf_1(3)^{tC_2}$ vanishes. Thus, we get:
\begin{cor}
There is an equivalence $I_{\mathbb{Z}\left[\tfrac13\right]} Tmf_0(3) \simeq (\Sigma^{5+2\rho}Tmf_1(3))^{hC_2}$.
\end{cor}

\section{The Picard Groups}\label{sec:Picard}
In this section we will compute the Picard groups of $TMF_0(3)$, $Tmf_0(3)$ and related spectra. We recommend \cite{M-S14} for a good introduction to Picard groups and our techniques are very similar to theirs.

\subsection{Generalities}\label{subsection:Generalities}
In the following, we will often use the language of $\infty$-categories. We choose the same model as \cite{JoyCRM} and \cite{HTT}, namely quasicategories. For the theory of (symmetric) monoidal $\infty$-categories see either \cite{HA} or \cite{Gro10} for a shorter introduction.

If $\CC$ is a monoidal category, we denote by $\Pic(\CC)$ the group of isomorphism classes of invertible spectra; note that this is a priori a proper class (or large set, depending on set-theoretic conventions), but will always be a (small) set in our situation. If $\CC$ is a monoidal $\infty$-category, we denote by $\mathcal{P}ic(\CC)$ the maximal $\infty$-subgroupoid (Kan complex) of the full subcategory of invertible objects. Clearly, $\pi_0\mathcal{P}ic(\CC) \cong \Pic(\Ho(\CC))$. If $\CC$ is a symmetric monoidal $\infty$-category, $\mathcal{P}ic(\CC)$ inherits the structure of a group-like $E_\infty$-space; indeed, $\mathcal{P}ic(\CC)$ is a symmetric monoidal $\infty$-category and thus by \cite[Ex 2.1.2.18, Rem 2.4.2.6, Cor 5.1.1.5]{HA} a $\mathrm{Comm} = E_\infty$-algebra in the $\infty$-category of $\infty$-groupoids, which agrees with that of spaces. Thus, there is a connective spectrum $\mathfrak{pic}(\CC)$ with $\Omega^{\infty}\mathfrak{pic}(\CC) \simeq \mathcal{P}ic(\CC)$ by a result of Boardman--
Vogt and May (see \cite[Rem 5.2.6.26]{HA} for an $\infty$-categorical treatment). Note that we have $\pi_i \mathfrak{pic}(\CC) \cong 
\pi_i \mathcal{P}ic(\CC)$ in this situation.

Given an $E_2$-ring spectrum $R$, its $\infty$-category $R\modules$ of (left) $R$-modules has the structure of a monoidal $\infty$-category (\cite[Proposition 7.1.2.6]{HA}). We define the \emph{Picard group} $\Pic(R)$ of $R$ to be $\Pic(\Ho(R\modules))$ and the \emph{Picard space} $\mathcal{P}ic(R)$ to be $\mathcal{P}ic(R\modules)$. If $R$ is an $E_\infty$-ring spectrum, then $R\modules$ is even a \emph{symmetric} monoidal $\infty$-category. We define then $\mathfrak{pic}(R)$ to be $\mathfrak{pic}(R\modules)$.

For us, a \emph{derived stack} will be a pair $\XX = (X,\OO^{top})$, where $X$ is a Deligne--Mumford stack and $\OO^{top}$ is a sheaf of even-periodic $E_\infty$-ring spectra with $\pi_0\OO^{top}$ isomorphic to the structure sheaf $\OO_X$ of $X$. For example, $X$ might be a moduli stack of elliptic curves. For a derived stack $\XX = (X,\OO^{top})$, we write $\Pic(\XX)$ etc.\ for the Picard group, space or spectrum of the symmetric monoidal $\infty$-category of quasi-coherent $\OO^{top}$-modules $\QCoh(\XX)$ on $\XX$. For a short treatment of quasi-coherent sheaves in this context see \cite[Section 2.3]{MM15} and for a full-blown treatment see \cite{DAGQC}.

\begin{defi}
We call a derived stack $\XX = (X,\OO^{top})$ \emph{$0$-affine} if the global sections functor
\[\Gamma\co \QCoh(\XX) \to \OO^{top}(X)\modules \]
is an equivalence of symmetric monoidal $\infty$-categories.
\end{defi}

Clearly, $\mathfrak{pic}(\XX) \simeq \mathfrak{pic}(\OO^{top}(X))$ if $\XX$ is $0$-affine. It was shown in \cite{MM15} that the (compactified) moduli stack of elliptic curves with arbitary level structure together with its derived structure sheaf $\OO^{top}$ is $0$-affine.

The following Mayer--Vietoris principle will be useful later.

\begin{lemma}\label{lem:MV}
Let $\XX = (X,\OO_\XX)$ be a $0$-affine derived stack and $U,V\subset X$ be a covering by open substacks. Then we have a long exact sequence
\[\cdots \to GL_1 \pi_0\OO_\XX(U\cap V) \xrightarrow{\partial} \Pic(\OO_\XX(X)) \to \Pic(\OO_\XX(U)) \times \Pic(\OO_\XX(V)) \to \Pic(\OO_\XX(U\cap V))\]
of abelian groups.
\end{lemma}
\begin{proof}
As shown in \cite[Section 3.1]{M-S14}, the presheaf $\mathcal{P}ic$ defined by
\[ \mathcal{P}ic(W \to X) = \mathcal{P}ic(\OO^{top}(W \to X))\]
(where $W \to X$ is \'etale) is actually a sheaf. Thus, we have a homotopy pullback square
\[\xymatrix{\mathcal{P}ic(X,\OO_\XX) \ar[r]\ar[d] & \mathcal{P}ic(U, \OO_\XX|_U) \ar[d] \\
\mathcal{P}ic(V,\OO_\XX|_V) \ar[r] & \mathcal{P}ic(U\cap V,\OO_\XX|_{U\cap V}) }
\]
The identification of these Picard spaces with those of $\OO_\XX(X)$ etc.\ follows from the fact that $X$, $U$, $V$ and $U\cap V$ are $0$-affine (see \cite[Proposition 3.27]{MM15}). This fiber square induces the long exact sequence in the lemma. 
\end{proof}
\begin{remark}By the last proof the boundary map 
$$GL_1 \pi_0\OO_\XX(U\cap V) \to \Pic(\OO_\XX(X))$$
is induced by the map 
\[GL_1\OO_\XX(U\cap V) \simeq \Omega \mathcal{P}ic(\OO_\XX(U\cap V)) \to \mathcal{P}ic(\OO_\XX(U)) \times_{\mathcal{P}ic(\OO_\XX(U\cap V))}^h \mathcal{P}ic(\OO_\XX(V))\]
of spaces. Thus, it can be described as follows: An element $g\in GL_1 \pi_0\OO_\XX(U\cap V)$ induces an $\OO_\XX$-linear self-equivalence $f$ of $\OO_\XX|_{U\cap V}$. The triple $(\OO_{\XX}|_U, \OO_{\XX}|_V, f)$ defines an element of the homotopy fiber product $\mathcal{P}ic(\OO_\XX(U)) \times_{\mathcal{P}ic(\OO_\XX(U\cap V))}^h \mathcal{P}ic(\OO_\XX(V))$. As noted above, this gluing datum defines an invertible $\OO_\XX$-module on $\XX$ and this invertible module represents $\partial(g)$. 
\end{remark}

Let now $A\to B$ be a faithful $G$-Galois extension in the sense of \cite{Rog08}. Then by \cite[Section 3.3]{M-S14}, we have the following theorem:
\begin{thm}\label{thm:picGalois}
There is an equivalence $\mathfrak{pic}(A) \simeq \tau_{\geq 0} \mathfrak{pic}(B)^{hG}$.
\end{thm}

There is also another equivariant interpretation of the Picard group of $A$ if $A\to B$ is a faithful $G$-Galois extension. View $B \simeq F(EG_+, B)$ as a cofree $G$-spectrum.  Denote the category of equivariant $B$-modules by $G\mhyphen B\modules$. As $B$ is cofree and $A\to B$ is a faithful Galois extension, $\Phi^GB \simeq B^{tG}$ is contractible. By \cite[Corollary 10.6]{HHR09} every (equivariant) $B$-module is thus cofree again. Therefore, a map in $G\mhyphen B\modules$ is a weak equivalence if it is an underlying weak equivalence. It is then a consequence of Galois descent in the form \cite[Lemma 6.1.4, Proposition 6.2.6]{Mei12} that there is a monoidal equivalence $\Ho(A\modules) \simeq \Ho(G\mhyphen B\modules)$. Thus, $\Pic(R) \cong \Pic(\Ho(G\mhyphen B\modules))$, the group of equivariant invertible $B$-modules. We will denote the latter group by $\Pic_G(B)$.

\subsection{A generalized Baker--Richter theorem}
Baker and Richter proved in \cite{B-R05} that the Picard group of an $E_\infty$-ring spectrum $R$ is completely algebraic if $R$ is even periodic and $\pi_0R$ is a regular complete local ring. This applies, for example, to the Lubin--Tate spectra $E_n$. Mathew and Stojanoska generalized this in \cite{M-S14} by dropping the condition that $\pi_0R$ is complete and local (and also weakened the periodicity requirement). The main purpose of this subsection is to show that the assumption of periodicity is superfluous.

Let $R$ be an $E_2$-ring spectrum. Let $\overline{L}$ be an invertible $\pi_*R$-module. Then $\overline{L}$ is projective over $\pi_*R$. Thus,  there is an $R$-module $L$ with $\pi_* L \cong \overline{L}$ and this module $L$ is well-defined up to isomorphism in $\Ho(R\modules)$. This defines a map $\Pic(\pi_*R) \to \Pic(R)$. By the degenerate K\"unneth spectral sequence, this is a homomorphism. 

Let $R_*$ be a commutative graded ring. By an element $x\in R_*$ we will always mean a \emph{homogeneous} element and by an ideal $I\subset R_*$ we will always mean a \emph{homogeneous} ideal. We call $R_*$ \emph{local} if it has a unique maximal ideal $\mathfrak{m}$. We call a graded local ring \emph{regular} if the maximal ideal is generated by a finite regular sequence. We call a graded local ring \emph{complete} if the map $R_* \to \lim_k R/\mathfrak{m}^k$ is an isomorphism. We call an arbitrary commutative graded ring \emph{regular} if every localization of it at a prime ideal is regular.

We have the following generalization of \cite[Theorem 38]{B-R05}.
\begin{thm}\label{thm:BakerRichter}
Let $R$ be an $E_2$-ring spectrum. Assume that $\pi_*R$ is concentrated in even degrees and regular. Then the morphism $\Pic(\pi_*R) \to \Pic(R)$ is an isomorphism.
\end{thm}
This is not really new as this generalization is just a combination of \cite[Remark 39]{B-R05} and \cite[Theorem 2.4.6]{M-S14}. We will sketch a proof anyhow as we introduce one simplification, avoiding the use of obstruction theory for $A_\infty$-structures.

Let $M$ be an invertible $R$-module with $M\sm_R N \simeq R$ for some $R$-module $N$. It is enough to show that $\pi_*M$ is a projective $\pi_*R$-module. For this, it is enough to show that the completion $\widehat{(\pi_*M)}_\mathfrak{m}$ is a projective $\widehat{\pi_*R}_\mathfrak{m}$-module for every maximal ideal $\mathfrak{m}\subset \pi_*R$.

\newcommand{\M}{\widehat{M}_\mathfrak{m}}
\renewcommand{\N}{\widehat{N}_\mathfrak{m}}
\renewcommand{\R}{\widehat{R}_\mathfrak{m}}

The theory from \cite[Section 4.2]{DAGXII} implies that there is actually an $R$-module $\M$ with $\pi_*\M \cong \widehat{(\pi_*M)}_\mathfrak{m}$.\footnote{Lurie only considers ideals in $\pi_0R$, but the theory also works for homogeneous ideals in $\pi_*R$ under our assumptions.}  We have $\M \sm_{\R} \N \simeq \R$ by \cite[Remark 4.2.6]{DAGXII} and because $N$ is a finite $R$-module. Note here that $\R$ also inherits an $E_2$-structure.

Let $x_1,\dots, x_n$ be a regular sequence generating $\mathfrak{m}$. Consider the $\R$-module 
$$\R/\underline{x} = \R/x_1\sm_{\R} \dots \sm_{\R} \R/x_n,$$
obtained by killing the regular sequence $x_1,\dots, x_n$. As $\R$ is even, every $x_i$ acts trivially on $\R/x_i$ and hence on $\R/\underline{x}$. Indeed, the composite 
$$\Sigma^{|x_i|}\R \to \Sigma^{|x_i|}\R/x_i \xrightarrow{\cdot x_i} \R/x_i$$
is zero and thus the second map factors over an $\R$-linear map $\Sigma^{2|x_i|+1}\R \to \R/x_i$, which must be zero as well. 

By \cite[Theorem V.2.6]{EKMM}\footnote{While the source states the result only for $E_\infty$-ring spectra, the same proof works also for $E_2$-ring spectra.} $\R/\underline{x}$ has the structure of an $\R$-ring spectrum in the sense that there exists a map
\[\R/\underline{x} \sm_{\R} \R/\underline{x} \to \R/\underline{x}\]
 that is unital up to homotopy.\footnote{For our argument, this naive result suffices, while Baker and Richter use that $\R/\underline{x}$ has an $A_\infty$-structure.} For an arbitrary $\R$-module $X$, set $X/\underline{x} = X\sm_{\R}\R/\underline{x}$.

\begin{claim}The map
\[\pi_*(X_1/\underline{x}) \tensor_{\pi_*\R} \pi_*(X_2/\underline{x}) \to \pi_*((X_1/\underline{x}\sm_{\R} X_2/\underline{x}) \to \pi_*((X_1\sm_{\R} X_2)/\underline{x})\]
factors over a map
$$\pi_*(X_1/\underline{x}) \tensor_{\pi_*\R/\underline{x}} \pi_*(X_2/\underline{x}) \to \pi_*((X_1\sm_{\R} X_2)/\underline{x}),$$
which is an isomorphism for all $\R$-modules $X_1$ and $X_2$.
\end{claim}
\begin{proof}
It factors as every $x_i$ acts trivially on $X_1/\underline{x} =X_1\sm_{\R}\R/\underline{x}$.

The map is clearly an isomorphism if $X_1 =\R$. Both sides are homological in $X_1$ -- as $\pi_*(\R/\underline{x})$ is a graded field -- and compatible with arbitrary coproducts. Thus, it is an isomorphism for all $X_1\in\R\modules$.
\end{proof}

In particular $\pi_*(\M/\underline{x})$ is in the Picard group of $\pi_*(\R/\underline{x})$. Thus, $\pi_*(\M/\underline{x})$ is a free $\pi_*\R/\underline{x}$-module of rank $1$.

As in \cite{B-R05}, one can show that $\pi_*(\M/(x_1^{i_1}, \dots, x_n^{i_n}))$ is a cyclic $\pi_*\R$-module for $i_1,\dots, i_n\geq 1$, using the Nakayama lemma for graded rings. Using the completeness of $\pi_*\R$, one can show as in \cite{B-R05} that $\pi_*\M$ is a shift of $\pi_*\R$. In particular, $\pi_*\M$ is projective over $\pi_*\R$ as we wanted to show.

\subsection{The case of \texorpdfstring{$TMF_1(3)$ and $Tmf_1(3)$}{TMF13 and Tmf13}}
\begin{lemma}\label{lem:NonEquivariantPic}
We have isomorphisms
\begin{align*}
\Pic TMF_1(3) &\cong \Z/6 \\
\Pic tmf_1(3)[a_1^{-1}] &\cong \Z/2\\
\Pic tmf_1(3)[a_3^{-1}] &\cong \Z/6\\
\Pic tmf_1(3)[a_1^{-1}\ab_3^{-1}] &\cong \Z/2
\end{align*}
In all the cases, all the invertible modules are equivalent to suspensions of the ground ring spectrum.
\end{lemma}
\begin{proof}
We will just prove the lemma for $TMF_1(3)$ -- the other cases are analogous. By Theorem \ref{thm:BakerRichter}, 
$$\Pic TMF_1(3) \cong \Pic(\pi_*TMF_1(3)).$$
An evenly graded $\pi_{2*}TMF_1(3)$-module is an equivalent datum to a quasi-coherent sheaf on $\MM_1(3) \simeq \Spec \Z\big[\tfrac13\big][a_1,a_3][\Delta^{-1}]/\G_m$. Furthermore, an arbitrary graded $\pi_{*}TMF_1(3)$-module splits into an even and an odd part. Thus an invertible $\pi_{*}TMF_1(3)$-module has to be either completely even or completely odd. We have hence a short exact sequence
\[0 \to \Pic(\MM_1(3))\to \Pic(\pi_*TMF_1(3))\to \Z/2 \to 0,\]
where the first map corresponds to the inclusion of the even part and the map to $\Z/2$ indicates whether the invertible module is even or odd. 

Given a line bundle $\LL$ on $\MM_1(3)$, we can extend it to the weighted projective stacky line $\MMb_1(3)$. Indeed, by \cite[Lemma 3.2]{Mei15}, we can extend $\LL$ to a reflexive sheaf on $\MMb_1(3)$ and every reflexive sheaf of rank $1$ is a line bundle by \cite[Proposition 1.9]{Har80}. Every line bundle on a weighted projective stacky line is of the form $\OO(k)$ for some $k\in\Z$ as can be seen, for example, along the lines of \cite[Proposition 3.4]{Mei15}. As noted after Proposition \ref{prop:StackIdentification}, the line bundle $\OO(k)$ restricts to the (pullback of) the line bundle $\omega^{\tensor k}$ on $\MM_1(3)$. Thus, the map $\phi\colon \Z \to \Pic(\MM_1(3))$ sending $k$ to $\omega^{\tensor k}$ is surjective. 

It follows from the identification of $\MM_1(3)$ above that $H^0(\MM_1(3); \omega^{\tensor *}) \cong \Z[\tfrac13][a_1,a_3,\Delta^{-1}]$ with $\Delta = a_3^3(a_1^3-27a_3)$. As thus $a_3 \in H^0(\MM_1(3);\omega^{\tensor 3})$ is invertible on $\MM_1(3)$, it defines a trivialization of $\omega^{\tensor 3}$ and thus $\phi(3) = 0$. The resulting morphism
\[ \overline{\phi}\colon \Z/3 \to \Pic(\MM_1(3))\]
is an isomorphism as there is no invertible section of $H^0(\MM_1(3);\omega^{\tensor i})$ for $i=1$ or $2$. As the subgroup of $\Pic TMF_1(3)$ spanned by the $\Sigma^kTMF_1(3)$ is isomorphic to $\Z/6$, the lemma follows.
\end{proof}

\begin{prop}\label{prop:Galois2}
The extensions
\begin{align*}
TMF_0(3) &\to TMF_1(3) \\
(tmf_1(3)[\ab_1^{-1}])^{hC_2} &\to tmf_1(3)[\ab_1^{-1}] \\
(tmf_1(3)[\ab_3^{-1}])^{hC_2} &\to tmf_1(3)[\ab_3^{-1}] \\
(tmf_1(3)[\ab_1^{-1}\ab_3^{-1}])^{hC_2} &\to tmf_1(3)[\ab_1^{-1}\ab_3^{-1}] \\
\end{align*}
are faithful $C_2$-Galois extensions in the sense of Rognes.
\end{prop}
\begin{proof}
We obtain these maps of $E_\infty$-ring spectra by applying $\OO^{top}$ to the following $C_2$-Galois covers of stacks
\begin{align*}
\MM_1(3) &\to \MM_0(3) \\
D(a_1) &\to D(a_1)/C_2 \\
D(a_3) &\to D(a_3)/C_2 \\
D(a_1a_3) &\to D(a_1a_3)/C_2
\end{align*}
as follows from the results in Section \ref{sec:relationship}. Here, $D$ denotes the non-vanishing locus. By the main result of \cite{MM15}, the derived stack $(\MM_{ell},\OO^{top})$ is $0$-affine and by \cite[Proposition 3.29]{MM15} the same is true for the targets of the above four Galois covers. Then \cite[Theorem 5.6]{MM15} implies the result.
\end{proof}

\begin{thm}\label{Theorem:PeriodicPic}
We have isomorphisms
\begin{align*}
\Pic_{C_2}TMF_1(3) \cong \Pic\big(TMF_0(3)\big) &\cong \Z/48 \\
\Pic_{C_2}tmf_1(3)[\ab_1^{-1}] \cong \Pic((tmf_1(3)[\ab_1^{-1}])^{hC_2}) &\cong \Z/8\\
\Pic_{C_2}tmf_1(3)[\ab_3^{-1}] \cong \Pic((tmf_1(3)[\ab_3^{-1}])^{hC_2}) &\cong\Z/48\\
\Pic_{C_2}tmf_1(3)[\ab_1^{-1}\ab_3^{-1}] \cong \Pic((tmf_1(3)[\ab_1^{-1}\ab_3^{-1}])^{hC_2}) &\cong \Z/8
\end{align*}
In all the cases, all the (equivariant) invertible modules are equivalent to (integer) suspensions of the ground ring spectrum.
\end{thm}
\begin{proof}
We will only prove this in the first case. The other cases are similar. The first isomorphism follows directly from Proposition \ref{prop:Galois2} and the discussion at the end of the previous subsection.

In the following, we will denote by \emph{HFPSS} the homotopy fixed point spectral sequence for the $C_2$-action on $TMF_1(3)$ and differentials in it will be denoted by $d^{HF}$. We will always use the Adams convention that the $k$-th column consists of the groups $H^s(C_2;\pi_tTMF_1(3))$ with $k = t-s$.

We have $TMF_1(3) \simeq_{C_2} tmf_1(3)[\overline{\Delta}^{-1}]$ with $\overline{\Delta} = \ab_3^3(\ab_1^3-27\ab_3)$ by the results of Section \ref{sec:relationship}. As $\overline{\Delta}$ is a permanent cycle, this allows us to deduce from the results of Section \ref{sec:specseq} all differentials in the HFPSS. For example, $\gamma = \frac{\ab_3^4}{\overline{\Delta}}$ is a permanent cycle.

It is easy to see that the $(-1)$-st column of the HFPSS for $TMF_1(3)$ is in cohomological degrees $\leq 7$ isomorphic to  $\F_2[\gamma]\cdot b_3\oplus \F_2[\gamma]\cdot b_7$ with $b_3 = a_{\sigma}^3\ab_1u_{2\sigma}^{-1}$ of cohomological degree $3$ and $b_7 = a_{\sigma}^7\ab_3u_{2\sigma}^{-2}$ of degree $7$. Recall from Section \ref{sec:specseq} that $\ab_1, \ab_3$ and $a_\sigma$ are permanent cycles while $d_3^{HF}(u_{2\sigma}) = a_{\sigma}^3\ab_1$ and $d_7^{HF}(u_{2\sigma}^2) = a_{\sigma}^7\ab_3$. We have thus the differentials
\[d_3^{HF}(\gamma^kb_3) = \gamma^ka_{\sigma}^3\ab_1u_{2\sigma}^{-2}d_3^{HF}(u_{2\sigma}) = \gamma^kb_3^2\]
and
\[d_7^{HF}(\gamma^kb_7) = \gamma^ka_{\sigma}^7\ab_3u_{2\sigma}^{-4}d_7^{HF}(u_{2\sigma}^2) = \gamma^kb_7^2\]
in the HFPSS.

As $TMF_0(3) \to TMF_1(3)$ is a faithful $C_2$-Galois extension, Theorem \ref{thm:picGalois} implies an equivalence $\mathfrak{pic}(TMF_0(3)) \simeq \tau_{\geq 0} (\mathfrak{pic}(TMF_1(3)))^{hC_2}$. This gives the \emph{Picard spectral sequence}
\[H^s(C_2; \pi_t\,\mathfrak{pic}\, TMF_1(3))\]
that converges to $\pi_{t-s}\mathfrak{pic}\,TMF_0(3)$ for $t-s \geq 0$. Differentials in it will be denoted by $d^{Pic}$

The Picard group of $TMF_1(3)$ is $\Z/6$ by Lemma \ref{lem:NonEquivariantPic} and $GL_1\pi_0TMF_1(3)$ is isomorphic to $\Z\times \Z/2$, generated by $\tfrac13$ and $-1$. Thus:
\[\pi_t\,\mathfrak{pic}\, TMF_1(3) = \begin{cases}\Z/6 & \text{ for } t=0 \\
\Z \times \Z/2 & \text{ for } t = 1 \\
\pi_{t-1}TMF_1(3) & \text{ for } t\geq 2.\end{cases}\]

We are interested in the $0$-th column of the Picard spectral sequence. We have
\[H^0(C_2;\Z/6) = \Z/6\]
 and
 \[H^1(C_2; \Z\times \Z/2) \cong \Z/2;\]
  for $s\geq 2$ the $0$-th column of the Picard spectral sequence agrees with the $(-1)$-st column of the HFPSS. For an element $x$ in the $(-1)$-st column of the HFPSS, denote the corresponding element in the $0$-th column of the Picard spectral sequence by $\underline{x}$.

If $x \in E^s_{*,s}$ is in cohomological degree $s$, then $d_s^{Pic}(\underline{x}) = \underline{d_s^{HF}(x)+x^2}$ by \cite[Theorem 6.1.1]{M-S14}. For degree reasons, the first possible differential for $\underline{\gamma^kb_3}$ is a $d_3^{Pic}$ and this equals $\underline{(\gamma^k+\gamma^{2k})b_3^2}$. This is zero only if $k=0$. Likewise for degree reasons, the first possible differential for $\underline{\gamma^kb_7}$ is a $d_7^{Pic}$ and this equals $\underline{(\gamma^k+\gamma^{2k})b_7^2}$. This is again zero only if $k=0$, so that $\underline{b_3}$ and $\underline{b_7}$ are the only permanent cycles in the $0$-th column of the Picard spectral sequence in cohomological degrees $2\leq s \leq 7$.

It is easy to check that each element in the $(-1)$-st column of the HFPSS of cohomological degree $\geq 8$ either supports a $d_3$- or $d_7$-differential or is hit by a $d_3$- or $d_7$-differential from an element of degree $\geq 8$. By \cite[Comparison Tool 5.2.4]{M-S14}, this implies that all non-trivial elements in the $0$-th column of the Picard spectral sequence in cohomological degrees $\geq 8$ support non-trivial differentials or are hit by differentials.

Thus, $\Pic (TMF_0(3))$ has at most $6\cdot 2\cdot 2\cdot 2 = 48$ elements. We just need to show that the image of the morphism
\[\Z \to \Pic (TMF_1(3)),\quad k \mapsto \Sigma^k TMF_0(3)\]
has order $48$. This follows easily from the fact that $48$ is the smallest period of $\pi_*TMF_0(3)$ as $\Delta$ is not a permanent cycle in the HFPSS.
\end{proof}

\begin{lemma}
 Let $E$ be a strongly even $C_2$ $E_2$-ring spectrum. Then every even projective $\pi_*E$ module can be realized by a strongly even $E$-module in a unique way up to homotopy, giving in particular a well defined homomorphism
\[\Pic_{even}(\pi^e_*E) \to \Pic^{C_2}(E).\]
\end{lemma}
\begin{proof}
 Let $P$ be an even projective $\pi_*^eE$-module. We can write $P$ as the image of an idempotent endomorphism $f$ of a free even $\pi_*E$-module $F$. We can write $F = \bigoplus_I \pi^e_*\Sigma^{2n_i}E$. Define a free $E$-module $\F$ by $\F = \bigoplus \Sigma^{n_i\rho}E$. Because $E$ is strongly even, we have $\pi^e_*\F \cong F$ and we can lift $f$ to an idempotent endomorphism of $\F$, whose mapping telescope we denote by $\mathbb{P}$ -- this is the required realization of $P$. 
 
 If we have another strongly even $E$-module $\mathbb{P}'$ with $\pi_*\mathbb{P}' \cong P$ as $\pi^e_*E$-modules, we can lift the morphism $F \to P$ to an $E$-module morphism $\F \to \mathbb{P}'$ and further to a morphism $\mathbb{P} \to \mathbb{P}'$ that induces an isomorphism on $\pi^e_*$. By Lemma \ref{lem:regrep}, this is an equivalence. 
 
 Thus, we get a well-defined map 
 \[\Pic_{even}(\pi^e_*E) \to \Pic^{C_2}(E).\]
 To show that it is an homomorphism, we have to show that for strongly even projective $E$-modules $\mathbb{P}_1$ and $\mathbb{P}_2$, the smash product $\mathbb{P}_1\sm_E \mathbb{P}_2$ is still strongly even and has underlying homotopy groups $\pi_*^e\mathbb{P}_1\tensor_{\pi_*^eE}\pi_*^e\mathbb{P}_2$. This is clear by a retraction argument from the corresponding statement for free modules of the form $\bigoplus_{i\in I}\Sigma^{n_i\rho}E$. 
\end{proof}

\begin{question}\label{q:Pic}
Let $E$ be a $C_2$ $E_2$-ring spectrum. Assume that $E$ is strongly even and that $\pi^u_*E$ is a regular graded ring and an integral domain. Is every invertible $E$-module of the form $S^V\sm L$, where $V\in RO(C_2)$ and $L$ is a strongly even $E$-module with $\pi_*^eL \in \Pic(\pi^e_*E)$?

Using the lemma above, the question can be restated as asking for the surjectivity of the homomorphism
\[
RO(C_2) \oplus \Pic_{even}(\pi^e_*E) \to \Pic^{C_2}(E).
\] 
A positive answer to this question would be a Real generalization of the theorem by Baker and Richter given here as Theorem \ref{thm:BakerRichter}.
\end{question}

We could provide a similar spectral sequence argument as above for the computation of $\Pic_{C_2}\big(Tmf_1(3)\big)$, but we prefer to use a Mayer--Vietoris style argument instead. This will demonstrate how the computation of $\Pic_{C_2}\big(Tmf_1(3)\big)$ follows essentially formally from the fact that the Picard groups $\Pic_{C_2}(tmf_1(3)[\ab_1^{-1}])$ and $\Pic_{C_2}(tmf_1(3)[\ab_3^{-1}])$ are generated by the suspension of the ground ring spectrum.

\begin{thm}
The morphism
\[RO(C_2) \to \Pic_{C_2}\big(Tmf_1(3)\big),\quad V \mapsto S^V\sm Tmf_1(3)\]
is surjective. Its kernel is generated by $8-8\sigma$. Thus,
\[\Pic\big(Tmf_0(3)\big) \cong \Pic_{C_2}\big(Tmf_1(3)\big) \cong \Z\oplus \Z/8.\]
\end{thm}
\begin{proof}
By Lemmas \ref{lem:invert}, \ref{lem:invert2} and \ref{lem:MV}, we have an exact sequence:

\vspace*{0.07cm}
\hspace*{-4.45cm}
\scalebox{.89}{
\begin{tikzpicture}
\node (g1) {$GL_1\pi_0^{C_2}tmf_1(3)[\ab_1^{-1}]\times GL_1\pi_0^{C_2}tmf[\ab_3^{-1}]$};
\node (g13) [right=of g1] {$GL_1 \pi_0^{C_2} tmf_1(3)[\ab_1^{-1}\ab_3^{-1}]$};
\node (p1) [below=of g1] {$\Pic_{C_2}(tmf_1(3)[\ab_1^{-1}]) \times \Pic_{C_2}(tmf_1(3)[\ab_3^{-1}])$};
\node (pT) [left=of p1] {$\Pic_{C_2}\big(Tmf_1(3)\big)$};
\node (p13) [right=of p1] {$\Pic_{C_2}(tmf_1(3)[\ab_1^{-1},\ab_3^{-1}])$};
\draw[->]
(g1) edge node[auto] {$f$} (g13)
(g13) edge[out=-20,in=155] node[pos=0.55,yshift=-7pt] {$\partial$} (pT)
(pT) edge (p1)
(p1) edge node[auto] {$g$} (p13);
\end{tikzpicture}
}
\vspace*{0.05cm}

By Corollary \ref{cor:HomotopyGroups}, we have $\pi_0^{C_2}tmf_1(3)[\ab_1^{-1}\ab_3^{-1}] \cong \Z\big[\tfrac13\big][(\ab_1^3\ab_3^{-1})^{\pm 1}]$. Thus,
\[GL_1 \pi_0^{C_2} tmf_1(3)[\ab_1^{-1}\ab_3^{-1}] \cong \Z\times \Z\times \Z/2,\] generated by $\tfrac13$, $\ab_1^3\ab_3^{-1}$ and $-1$ and $\coker(f) \cong \Z$ generated by $[\ab_1^3\ab_3^{-1}]$.

We claim that $\partial(\ab_1^3\ab_3^{-1}) \simeq S^{3\rho}\sm Tmf_1(3)$. Indeed, we have trivializations \[\ab_3\co  S^{3\rho}\sm tmf_1(3)[\ab_3^{-1}] \to  tmf_1(3)[\ab_3^{-1}]\]
and
\[\ab_1^3\co S^{3\rho}\sm tmf_1(3)[\ab_1^{-1}] \to  tmf_1(3)[\ab_1^{-1}].\] Thus, we get $S^{3\rho}\sm Tmf_1(3)$ by gluing $tmf_1(3)[\ab_3^{-1}]$ and $tmf_1(3)[\ab_1^{-1}]$ by the map $\ab_1^3\ab_3^{-1}$ on $tmf_1(3)[\ab_1^{-1}\ab_3^{-1}]$.

By Theorem \ref{Theorem:PeriodicPic}, $\ker(g) \cong \Z/48$. Furthermore, $\Sigma^{8-8\sigma} Tmf_1(3) \simeq_{C_2}Tmf_1(3)$ as $u_{2\sigma}^4$ is a permanent cycle. Thus, we get a commutative diagram

\[
\xymatrix{
0 \ar[r] & \Z\ar[d]^{\cong} \ar[r]^-{3\rho} & RO(C_2)/(8-8\sigma)\ar[d] \ar[r] & RO(C_2)/(8-8\sigma, 3\rho)\cong \Z/48 \ar[d]^{\cong} \ar[r] & 0
\\
0 \ar[r] & \coker(f) \ar[r] & \Pic_{C_2}\big(Tmf_1(3)\big) \ar[r] & \ker(g) \ar[r] & 0
}
\]

Thus, the middle map is also an isomorphism.
\end{proof}


\begin{remark}
The map $\Pic(Tmf) \to \Pic\big(Tmf_0(3)\big)$ is \emph{not} surjective. The former has been identified with $\Z\oplus \Z/24$ in \cite[Thm B, Constr 8.4.2]{M-S14}, where the summands are generated by $\Sigma Tmf$ and by the global sections $\OO^{top}(\mathcal{I})$. Here, $\mathcal{I}$ is a line bundle on the derived compactified moduli stack of elliptic curves $(\MMb_{ell}, \OO^{top})$ obtain by gluing $\Sigma^{24}\OO^{top}$ on $\MM_{ell}$ and $\Sigma^{24}\OO^{top}$ on $\MMb_{ell}[c_4^{-1}]$ via the clutching function $j= \frac{c_4^3}{\Delta}$.

We claim that the module $\OO^{top}(\mathcal{I})\sm_{Tmf}Tmf_0(3)$ is $2$-torsion in $\Pic\big(Tmf_0(3)\big)$. Indeed, for $p\colon \MMb_0(3)\to \MMb_{ell}$, we have for an arbitrary locally-free sheaf $\FF$ on $(\MMb_{ell}, \OO^{top})$ an equivalence
\begin{align*}
\Gamma(\FF)\sm_{Tmf}Tmf_0(3) &\simeq \Gamma(\MMb_{ell}; \FF \sm_{\OO^{top}} p_*\OO^{top}_{\MMb_0(3)}) \\
&\simeq \Gamma(\MMb_{ell}; p_*(p^*\FF \sm_{\OO^{top}_{\MMb_0(3)}} \OO^{top}_{\MMb_0(3)}) \\
&\simeq \Gamma(\MMb_0(3); p^*\FF)
\end{align*}
In the first equivalence, we use that $(\MMb_{ell}, \OO^{top})$ is $0$-affine and in the second we use the projection formula (see \cite[Ex II.5.1d]{Har77} for the algebraic statement, from which the topological can be deduced).
Thus, $\OO^{top}(\mathcal{I})\sm_{Tmf}Tmf_0(3)$ can be constructed as $\OO^{top}(p^*\mathcal{I})$, where $p^*\mathcal{I}$ can be constructed by an analogous gluing construction on $\MMb_0(3)$, gluing $\Sigma^{24}\OO^{top}$ on $\MM_0(3)$ and $\Sigma^{24}\OO^{top}$ on $\MMb_0(3)[c_4^{-1}]$ via the clutching function $j = \frac{c_4^3}{\Delta}$ with $c_4=a_1^4 - 24a_1a_3$. There is an equivalence of gluing data
\[(\OO^{top}, \OO^{top}, \id) \to (\Sigma^{48}\OO^{top},\Sigma^{48}\OO^{top}, j^2)\]
given by $\Delta^2\colon \OO^{top} \to \OO^{top}$ on $\MM_0(3)$ and $c_4^6\colon \OO^{top} \to \Sigma^{48}\OO^{top}$ on $\MMb_0(3)[c_4^{-1}]$. Note here that $\Delta^2 = \overline{\Delta}^2u_{2\sigma}^{12}$ is a permanent cycle for $TMF_0(3)$. Thus, $2\cdot [\mathcal{I}] = 0\in \Pic(\MMb_0(3), \OO^{top}) \cong \Pic\big(Tmf_0(3)\big)$.

As not every torsion in $\Pic\big(Tmf_0(3)\big)$ is $2$-torsion, $\Pic(Tmf) \to \Pic\big(Tmf_0(3)\big)$ is indeed not surjective.
\end{remark}

\bibliographystyle{alpha}
\bibliography{Chromatic}
\end{document}